\documentclass[12pt]{amsart}
\usepackage[latin1]{inputenc}
\usepackage[english]{babel}
\usepackage{amsmath, amsfonts, amssymb, amsthm}
\usepackage{enumerate}
\usepackage{nicefrac}
\usepackage[numbers]{natbib}

\usepackage[arrow,matrix,curve]{xy}
\usepackage[pagebackref,colorlinks,linkcolor=red,citecolor=blue,urlcolor=blue,hypertexnames=true]{hyperref}

\newtheorem{theorem}{Theorem}[section]
\newtheorem{lemma}[theorem]{Lemma}
\newtheorem{corollary}[theorem]{Corollary}

\newtheorem*{corNormal}{Corollary}
\newtheorem*{thmNoLabel}{Theorem}

\theoremstyle{definition}
\newtheorem{definition}[theorem]{Definition}
\newtheorem{example}[theorem]{Example}
\newtheorem{remark}[theorem]{Remark}

\newcommand{\N}{\mathbb{N}}
\newcommand{\Z}{\mathbb{Z}}

\newcommand{\C}{\mathbb{C}}
\newcommand{\F}{\mathbb{F}}		
\newcommand{\K}{\mathbb{K}}		

\newcommand{\scalA}[2]{\,_A\langle #1,#2 \rangle}
\newcommand{\scalB}[2]{\langle #1,#2 \rangle_B}
\newcommand{\scalR}[3]{\langle #1,#2 \rangle_{#3}}
\newcommand{\scalL}[3]{\,_{#1}\langle #2,#3 \rangle}
\newcommand{\norm}[1]{\lVert #1 \rVert}
\newcommand{\cpthom}[4]{\mathcal{K}_{#1{\rm -}#2}(#3,#4)}
\newcommand{\cptEnd}[3]{\mathcal{K}_{#1{\rm -}#2}(#3)}
\newcommand{\mmm}[3]{{\rm end}_{#1{\rm -}#2}(#3)}
\newcommand{\bddHom}[4]{\mathcal{L}_{#1{\rm -}#2}(#3, #4)}
\newcommand{\bddEnd}[3]{\mathcal{L}_{#1{\rm -}#2}(#3)}
\newcommand{\U}[3]{U_{#1{\rm -}#2}(#3)}
\newcommand{\HA}[1]{\mathbb{H} \otimes #1}
\newcommand{\bibddEnd}[3]{\mathcal{L}_{#1#2}(#3)}

\newcommand{\biU}[3]{U_{#1#2}(#3)}
\newcommand{\twA}[1]{\!\,_{#1} A}
\newcommand{\Atw}[1]{A_{#1}}
\newcommand{\rtw}[2]{#1_{#2}}
\newcommand{\ltw}[2]{\!\,_{#2} #1}
\newcommand{\trivial}[1]{\underline{#1}}
\newcommand{\Hb}{\mathcal{H}}
\newcommand{\Ab}{\mathcal{A}}
\newcommand{\Hbop}{\mathcal{H}^{\rm op}}
\newcommand{\Nb}{\mathcal{N}}
\newcommand{\Qb}{\mathcal{Q}}
\newcommand{\Pb}{\mathcal{P}}
\newcommand{\Eb}{\mathcal{E}}
\newcommand{\Fb}{\mathcal{F}}
\newcommand{\Cuntz}[1]{\mathcal{O}(#1)}
\newcommand{\Aut}[1]{{\rm Aut}(#1)}
\newcommand{\Inn}[1]{{\rm Inn}(#1)}
\newcommand{\Pic}[1]{{\rm Pic}(#1)}
\newcommand{\id}[1]{{\rm id}_{#1}}
\newcommand{\Tw}[2]{\textsc{Twists}(#1,#2)}

\title{Twisted $K$-theory with coefficients in $C^*$-algebras }
\author{Ulrich Pennig}
\date{\today}

\begin{document}

\begin{abstract}
We introduce a twisted version of $K$-theory with coefficients in a $C^*$-algebra~$A$, where the twist is given by a new kind of gerbe, which we call Morita bundle gerbe. We use the description of twisted $K$-theory in the torsion case by bundle gerbe modules as a guideline for our noncommutative generalization. As it turns out, there is an analogue of the Dixmier-Douady class living in a nonabelian cohomology set and we give a description of the latter via stable equivalence classes of our gerbes. We also define the analogue of torsion elements inside this set and extend the description of twisted $K$-theory in terms of modules over these gerbes. In case $A$ is the infinite Cuntz algebra, this may lead to an interpretation of higher twists for $K$-theory. 
\end{abstract}
\maketitle

\section{Introduction}
Twisted $K$-theory first appeared in the early 1970s in a paper by Donovan and Karoubi as "K-theory with local coefficients" \cite{paper:DonovanKaroubi, paper:AtiyahSegal, paper:AtiyahSegal2}. It has by now become not only a well-understood example of a parametrized cohomology theory \cite{book:MaySigurdsson}, but also plays a fundamental role in the description of $D$-branes in string theory \cite{paper:DBranes}. 

The most prominent twists for $K$-theory over a space $M$ (which we will assume to be a compact manifold for simplicity) are classified by the third cohomology group $H^3(M,\Z)$, which can be seen as a $BBU(1)$-factor of the space $BBU_{\otimes}$ classifying all twists. A geometric description of these is given by $(S^1)$-bundle gerbes developed by Murray \cite{paper:BundleGerbes}, which are higher algebraic versions of line bundles and are up to an equivalence relation called stable isomorphism classified precisely by the third cohomology group via the Dixmier-Douady class of the gerbe \cite{paper:KTheoryBGM}. Bundle gerbes allow an easy description of connections on them, there is a Chern character for bundle gerbe modules, a Thom isomorphism and push-forward maps as well \cite{paper:CareyWang}.

The present paper takes a step towards a noncommutative generalization of twisted $K$-theory using bundle gerbes as a guideline. It was motivated by index theory: Let $A$ be a (unital) $C^*$-algebra, then the operator algebraic $K$-group $K_0(C(M,A))$ may be viewed as $K$-theory with coefficients in $A$ in the sense that it is the Grothendieck group of isomorphism classes of (finitely generated and projective) Hilbert $A$-module bundles over $M$. It appears in connection with important geometric invariants like the Rosenberg index \cite{paper:Rosenberg, paper:StableGLR, paper:MishchenkoFomenko}. In analogy with the case $A = \C$ we may ask whether it is possible to twist this theory and what the twists should look like. We answer the second question by a comparison with gerbes: Let $U_i$ be a good cover of a manifold $M$, then a gerbe provides a Morita self-equivalence of $C_0(U_{ij})$ in form of a line bundle $L_{ij} \to U_{ij}$ together with a bimodule isomorphism
\[
	C_0(U_{ijk},L_{ij}) \otimes_{C_0(U_{ijk})} C_0(U_{ijk}, L_{jk}) \to C_0(U_{ijk}, L_{ik})
\]
over the triple intersections. Speaking loosely, we may think of the line bundle as a pa\-ra\-me\-trized version of a pointwise $\C$-$\C$-Morita equivalence. So, what we need in case of $K$-theory with coefficients in $A$ is a parametrized version of $A$-$A$-Morita equivalences, which will be called twisted Morita bundles in this paper (see definition \ref{def:twistedMorita}). From there it is straightforward to define Morita bundle gerbes (see definition \ref{def:MoritaBG}), for which there are plenty examples: In fact, every continuous Busby-Smith twisted $G$-action on $A$ together with a principal $G$-bundle over $M$ yields a Morita bundle gerbe (see \cite{paper:BusbySmith} for the definition of these actions). Using the analogue of stable isomorphism, we define the set $\Tw{M}{A}$. After stabilizing the algebra, i.e.\ replacing $A$ by $A_s = A \otimes \K$ there is also an analogue of the Dixmier-Douady class, which now lives in the nonabelian cohomology set $\check{H}^1(M, U(A_s) \to \Aut{A_s})$, where $A_s = A \otimes \K$ and $U(A_s) \to \Aut{A_s}$ is a topological crossed module \cite{paper:Noohi}. It is in bijection with $\Tw{M}{A}$ and in case $A =\C$, it boils down to the third cohomology group and our class coincides with the previous one. Since $\Pic{A} = \Aut{A_s} / \Inn{A_s}$ we see a close relation between the Picard group and our set of twists. In fact, we would like to promote the picture to think about $\Pic{A}$ more in terms of the strict topological $2$-group associated to the crossed module $U(A_s) \to \Aut{A_s}$. Along the way, we prove the following structural result about $\Pic{A}$:
\begin{corNormal}
Let $A$ be a unital $C^*$-algebra and suppose $\Aut{A}$ is path-connected in the point-norm topology. Then the image of $\Aut{A}$ in $\Pic{A}$ is a normal subgroup. 
\end{corNormal}
The connectedness assumption is sufficient, but not necessary for ${\rm Out}(A)$ to be normal in $\Pic{A}$. For simple $C^*$-algebras carrying a unique trace which separates the projections in $A$, the quotient was identified with the $C^*$-algebraic analogue of the fundamental group in~\cite{paper:NawataWatatani}. It would be interesting to see connections between these two results.

Let ${\rm Fred}(A)$ be the Fredholm operators on $\HA{A}$ for a separable infinite dimensional Hilbert space $\mathbb{H}$, then
\[
	K_0(C(M,A)) \ \simeq \ [M, {\rm Fred}(A)]
\]
Like $PU(H)$ acts on ${\rm Fred}(H)$, we now have the group $\Aut{A_s}$ acting on ${\rm Fred}(A)$ by conjugation, since any automorphism $\sigma$ of $A_s$ extends to an automorphism of $M(A_s) = \bddEnd{}{A}{\HA{A}}$ and $\sigma \circ K \circ \sigma^{-1}$ is compact if $K$ is compact. Thus, twists should be representable by principal $\Aut{A_s}$-bundles. In fact, we apply the Packer-Raeburn stabilization trick to show that there is a surjective map $H^1(M, \Aut{A_s}) \to \Tw{M}{A}$. Thus, for every twist we get a (possibly non-unique) bundle of $C^*$-algebras affiliated with it.

Every Morita bundle gerbe with typical fiber $A$ gives rise to a nonabelian bundle gerbe in the sense of Aschieri, Cantini and Jur\v{c}o \cite{paper:JurcoNonAbelian} and vice versa. We connect our work to theirs. However, if the typical fiber varies along the connected components of the space, Morita bundle gerbes contain more information than just the $2$-group $U(A) \to \Aut{A}$. 

In analogy with the case of bundle gerbe modules, we then define twisted operator $K$-theory $K^0_{\Hb}(M)$ as the Grothendieck group of isomorphism classes of twisted Morita modules with respect to a twist $\Hb$ (definition \ref{def:MoritaModule}). This group only depends on the stable equivalence class of the twist, i.e.\ the element $[\Hb] \in \Tw{M}{A}$. Even though the latter is just a pointed set and not a group, it is still possible to give an analogue of torsion classes as the so called matrix stable elements, which lie in the image of the map $H^1(M, \Aut{M_n(A)}) \to \Tw{M}{A}$. We extend the operator algebraic description of twisted $K$-theory \cite{paper:RosenbergCTrace, paper:KTheoryBGM} to twisted operator $K$-theory by proving
\begin{thmNoLabel}
Let $\Hb$ be a representative of a matrix stable element $[\Hb] \in \Tw{M}{A}$ and let $\Ab \to M$ be a bundle of $C^*$-algebras affiliated with $\Hb$. Then we have
\[
	K^0_{\Hb}(M) \ \simeq\ K_0(C(M,\Ab))\ .
\]
\end{thmNoLabel}
In the last section we consider the concrete example of Cuntz algebras $\Cuntz{n}$, for which at least the automorphism group is quite well understood. This way, we get twisted $K$-theory with coefficients in $\Z/(n-1)\Z$. The construction also applies to the $C^*$-algebra $\Cuntz{\infty}$ and our conjecture would be that the automorphisms of $\Cuntz{\infty}$ are related to the higher twists of (ordinary) $K$-theory, which arise from the full space $BBU_{\otimes}$ not only the $BBU(1)$-factor.

\section{Preliminaries about $C^*$-algebras and Hilbert $C^*$-modules}
This short section will review some basics about Hilbert $C^*$-(bi)modules and Morita equivalences without giving proofs. Details can be found in the book by Lance \cite{book:Lance} and in the paper by Paschke \cite{paper:Paschke}.

\begin{definition}\label{def:auttop}
Let $A$ be a $C^*$-algebra. $\Aut{A}$ will denote the group of $*$-automorphisms of $A$ equipped with the pointwise norm topology, generated by the semi-norms $p_a(\varphi) = \lVert \varphi(a) \rVert$. It is a topological group with respect to this topology.
\end{definition}
Let $U$ be a topological space and consider a continuous map $\Psi \colon U \times A \to A$, where $A$ is of course equipped with its norm topology, such that $\Psi$ is a $*$-automorphism at every point. It is easy to check that this induces a continuous map $\varphi \colon U \to \Aut{A}$. Using the fact that $*$-automorphisms preserve the $C^*$-norm it can be shown that the converse of this is also true, i.e.\ $\varphi \colon U \to \Aut{A}$ translates into a continuous map $\Psi \colon U \times A \to A$.

\begin{definition}\label{def:HilbMod}
Let $A$ be a $C^*$-algebra. A left $A$-module $H$ is called a \emph{Hilbert $A$-module} if it is equipped with an $A$-valued inner product $\scalA{\cdot}{\cdot}$, which satisfies the following conditions:
\begin{enumerate}[a)]
	\item $\scalA{\cdot}{\cdot}$ is $A$-linear in the left argument and $A$-antilinear in the right.
	\item $\scalA{x}{x} \geq 0$ for all $x \in H$, where equality only holds if and only if $x = 0$.
	\item $\scalA{x}{y} = \scalA{y}{x}^*$.
	\item $H$ is complete with respect to the norm given by $\lVert x \rVert = \lVert \scalA{x}{x} \rVert^{\frac 1 2}$.
\end{enumerate}
Correspondingly we can define a right Hilbert $B$-module with the only change, that the inner product $\scalB{\cdot}{\cdot}$ should now be $B$-linear in the right argument and anti-linear in the left. If $H$ and $H'$ are two left Hilbert $A$-modules we denote by $\bddHom{A}{}{H}{H'}$ the \emph{bounded adjointable} operators $T \colon H \to H'$ and we define $\bddEnd{A}{}{H}= \bddHom{A}{}{H}{H}$. Similarly $\cpthom{A}{}{H}{H'}$ will denote the \emph{compact adjointable} left $A$-linear operators, likewise let $\cptEnd{A}{}{H}$ be the compact adjointable endomorphisms. If $H$ and $H'$ are right Hilbert $A$-modules, the corresponding spaces will be $\bddHom{}{A}{H}{H'}$, $\bddEnd{}{A}{H}$, $\cpthom{}{A}{H}{H'}$ and $\cptEnd{}{A}{H}$.
\end{definition}

\begin{definition}\label{def:unitarygroup}
The \emph{unitary group} of a left Hilbert $A$-module $H$ is the group of isometric isomorphisms $u \colon H \to H$ equipped with the \emph{strict topology} generated by the semi-norms $p_v(u) = \lVert u(v) \rVert$, $p_v^*(u) = \lVert u^*(v) \rVert$ for all $v \in H$ with respect to which it is a topological group. We will use the notation $\U{A}{}{H} = U(\bddEnd{A}{}{H})$ for it. Analogously for right Hilbert $A$-modules. 
\end{definition}

Just as for the group $\Aut{A}$, we can show that every continuous map $U \times H \to H$, which is an isometric isomorphism at every point $x \in U$, yields a continuous map $U \to \U{A}{}{H}$ and vice versa. Note that $U(A) = \U{A}{}{A}$ is the unitary group in the multiplier algebra $M(A) = \bddEnd{A}{}{A}$ of $A$.

\begin{remark}\label{rem:carefully}
A word of caution about the topologies on $U(A)$ and $\Aut{A}$ seems appropriate: The canonical map 
\[
	U(A) \to \Aut{A}
\]
sending a unitary to the corresponding inner automorphism is continuous with kernel $ZU(A)$ the center of $U(A)$, but the induced bijection on the quotient $U(A)/ZU(A) \to \Inn{A}$ is a homeomorphism if and only if $A$ is a continuous trace $C^*$-algebra \cite{paper:Phillips}.
\end{remark}

\begin{definition}\label{def:HilbBimod}
Let $A$ and $B$ be $C^*$-algebras. A \emph{Hilbert $A$-$B$-bimodule} is a vector space $H$, which is a right Hilbert $B$-module and a left Hilbert $A$-module, such that the $A$-, respectively $B$-valued inner products are compatible in the following sense:
\begin{equation} \label{eqn:switchscal}
\scalA{x}{y}\cdot z = x\cdot \scalB{y}{z} \quad \forall x,y,z \in H\ .
\end{equation}
A Hilbert $A$-$B$-bimodule $H$ is a \emph{Morita equivalence} if $A$ and $B$ act via compact operators and the left action map $A \to \cptEnd{}{B}{H}$ and the right action map $B^{\rm op} \to \cptEnd{A}{}{H}$ are isomorphisms of $C^*$-algebras. 
\end{definition}

\begin{remark} \label{rem:fingenproj}
If $A$ is a unital $C^*$-algebra and $H$ is an $A$-$B$-Morita equivalence, then the identity $\id{H} \in \cptEnd{}{B}{H}$ is a compact operator (it is the image of $1 \in A$). It follows that $\id{H}$ is in fact a finite rank operator, therefore $H$ is finitely generated and projective as a right Hilbert $B$-module. In particular, every $A$-$A$-Morita equivalence of a unital $C^*$-algebra $A$ is finitely generated and projective as a left as well as a right Hilbert $A$-module.
\end{remark}

\begin{definition}\label{def:conjugatebimodule}
Let $H$ be a Hilbert $A$-$B$-bimodule and let $H^*$ be the conjugate vector space of $H$. We will denote the elements of $H^*$ by $v^*$ for $v \in H$ just to distinguish them from elements of $H$. $H^*$ is equipped with the $B$-$A$-bimodule structure $a \cdot x^* \cdot b = (b^*xa^*)^*$ and the inner products $\scalR{x^*}{y^*}{A} = \scalA{x}{y}$ and $\scalL{B}{x^*}{y^*} = \scalB{x}{y}$. This way, $H^*$ becomes a Hilbert $B$-$A$-bimodule, which is called the bimodule \emph{conjugate} to $H$. 
\end{definition}

\begin{definition}\label{def:tensorproduct}
Let $H$ be a Hilbert $A$-$B$-bimodule, $K$ be a Hilbert $B$-$C$-bimodule. Their algebraic tensor product $H \otimes_B^{\rm alg} K$ over $B$ is an inner product bimodule, where the inner products on elementary tensors look like 
\begin{eqnarray*}
	\scalR{h_1 \otimes k_1}{h_2 \otimes k_2}{C} = \scalR{k_1}{\scalR{h_1}{h_2}{B}\cdot k_2}{C} \ ,\\
	\scalL{A}{h_1 \otimes k_1}{h_2 \otimes k_2} = \scalL{A}{h_1\cdot \scalL{B}{k_1}{k_2}}{h_2} \ .
\end{eqnarray*}
The completion with respect to the norm induced by either of the inner products is a Hilbert $A$-$C$-bimodule.
\end{definition}

\begin{remark}
As is well known, an $A$-$B$-Morita equivalence $H$ can be understood as generalized morphism between the $C^*$-algebras $A$ and $B$. If $K$ is a $B$-$C$-Morita equivalence, the composition of $H$ and $K$ is defined as $H \otimes_B K$. Let $I_A = \scalA{H}{H}$ be the closed linear span of all inner products. This is a closed two-sided ideal of $A$, therefore also a Hilbert $A$-$A$-bimodule, and $\scalA{\cdot}{\cdot}$ yields a bimodule isomorphism
\[
	T_A \colon H \otimes_B H^* \to I_A \quad , \quad x \otimes y^* \mapsto \scalA{x}{y}\ .
\]
Define $I_B$ analogously. Then we have $I_A \otimes_A H \simeq H$ and $H \otimes_B I_B \simeq H$ via the canonical maps. Thus, equation (\ref{eqn:switchscal}) is equivalent to 
\begin{equation}
	H \to H \otimes_B I_B \overset{\id{H} \otimes T_B^*}{\longrightarrow} H \otimes_B (H^* \otimes_A H) \to (H \otimes_B H^*) \otimes_A H \overset{T_A \otimes \id{H}}{\longrightarrow} I_A \otimes_A H \to H
\end{equation}
being the identity. Furthermore, equation (\ref{eqn:switchscal}) implies $I_A \simeq \cptEnd{}{B}{H}$ and $I_B \simeq \cptEnd{A}{}{H}$ as is shown for example in \cite[Proposition 1.10]{paper:QuasiMultAndEmb}. In case $H$ is a Morita equivalence, the inner products induce bimodule isomorphisms $H \otimes_B H^* \simeq A$ and $H^* \otimes_A H \simeq B$. 
\end{remark}

\begin{definition}\label{def:PicardGroup}
By the last remark, the isomorphism classes of  $A$-$A$-Morita equivalences form a group, which we will call the \emph{Picard group} $\Pic{A}$. A left Hilbert $A$-module $H$ is called \emph{full}, if $I_A=A$ with the notation from the previous remark.
\end{definition}

\begin{example} \label{ex:autTwist}
Let $H$ be a Hilbert $A$-$B$-bimodule and $\tau \in \Aut{B}$. Then we define $\rtw{H}{\tau}$ to be the induced Hilbert $A$-$B$-bimodule with the right $B$-module structure twisted by $\tau$, i.e.\ $v\cdot b = v\tau(b)$ and the $B$-linear inner product given by $\scalB{v}{w}^{\tau} = \tau^{-1}(\scalB{v}{w})$. The $A$-linear inner product is left unchanged. For $\sigma \in \Aut{A}$, let $\ltw{H}{\sigma}$ be defined analogously. In particular, the Hilbert $A$-$A$-bimodule $\twA{\sigma}$ is a Morita (self-)equivalence. Analogously, let $\Atw{\sigma}$ be the Hilbert $A$-$A$-module with the ordinary left action of $A$ and $\sigma$-twisted right action. The automorphism $\sigma$ induces a bimodule isomorphism $\Atw{\sigma^{-1}} \to \twA{\sigma}$. Moreover, this construction is compatible with the tensor product in the following sense: Let $\sigma, \tau \in \Aut{A}$. Then
\[
	\Atw{\sigma} \otimes_A \Atw{\tau} \to \Atw{\sigma \circ \tau}  \quad ; \quad x \otimes y \mapsto x\sigma(y)
\]
induces a Hilbert $A$-$A$-bimodule isomorphism. The situation for more general bimodules $H$ is more complicated, for example $\rtw{H}{\sigma^{-1}}$ is in general \emph{not} isomorphic to $\ltw{H}{\sigma}$. There are, however, canonical isomorphisms $(\rtw{H}{\sigma})^* \to \ltw{(H^*)}{\sigma}$ and $\rtw{H}{\tau} \otimes_B \rtw{B}{\sigma} \to \rtw{H}{\tau \circ \sigma}$.

\end{example}

Due to the two inner products, there are essentially two ways of forming the adjoint of an operator. The next lemma shows that in the case of \emph{endomorphisms}, they coincide, which allows us to define the unitary group of a Hilbert bimodule.
\begin{lemma}\label{lem:unitary}
Let $H$ be a Hilbert $A$-$B$-bimodule. If an $A$-$B$-bilinear operator $T \colon H \to H$ has an adjoint $T^{*,A}$ with respect to the $A$-valued inner product, then it also has an adjoint $T^{*,B}$ with respect to the $B$-valued inner product and $T^{*,A} = T^{*,B}$. Moreover the left multiplication by $a \in A$ is adjointable for the inner product $\scalB{\cdot}{\cdot}$ with adjoint $a^*$ and similarly for the right multiplication by $b \in B$.
\end{lemma}
\begin{proof}\label{pf:unitary}
For every $x,z \in H'$ and $y \in H$ we have
\begin{align*}
	x\cdot \scalB{Ty}{z} &= \scalA{x}{Ty}\cdot z = \scalA{T^*x}{y}\cdot z = T^*(x) \cdot \scalB{y}{z} = T^*(x \cdot \scalB{y}{z}) \\ 
	& = T^*(\scalA{x}{y} \cdot z) = \scalA{x}{y}\cdot T^*z = x\cdot \scalB{y}{T^*z}
\end{align*}
Note that $\alpha = \scalB{Ty}{z} - \scalB{y}{T^*z} \in I_B$, i.e.\ the two-sided closed ideal spanned by linear combinations of $B$-valued inner products. Choose an approximate unit $e_{\lambda}$ for $I_B$. Since $\alpha e_{\lambda} = 0$ by the previous calculation we have
\[
\norm{\alpha} = \lim_{\lambda}\norm{\alpha e_\lambda - \alpha} = 0\ . 
\]
A similar calculation with $\beta = \scalB{ay}{z} - \scalB{y}{a^*z} \in I_B$ proves the second assertion.
\end{proof}

\begin{definition}
By lemma \ref{lem:unitary} there is a well-defined notion of \emph{bounded adjointable endomorphism} of a Hilbert $A$-$B$-bimodule and the algebra of those will be denoted by $\bibddEnd{A}{B}{H}$. Moreover we define $\biU{A}{B}{H}$ to be the unitary group of this algebra.
\end{definition}

\section{Morita bundle gerbes}
Before we get to the definition of Morita bundle gerbes we need to clarify some technical issues concerning bimodules with bimodule structure parametrized by a space $U$. Throughout this chapter we will mostly consider unital $C^*$-algebras, even though we can formulate the technical theorems in a slightly more general way using the definition of $A$-finite rank Hilbert modules. We will often denote a trivial bundle with fiber $F$ by $\trivial{F}$.

\begin{definition}\label{def:Afinrank}
Let $B$ be a $C^*$-algebra. We will say that a (right) Hilbert $B$-module $H$ is of $B$-finite rank if the identity $\id{H}$ is in $\cptEnd{}{B}{H}$.
\end{definition}

The following theorem was proven by Exel in \cite{paper:Exel}
\begin{theorem}\label{thm:AfinRank}
	$H$ is of $B$-finite rank if and only if there exists an idempotent element $p \in M_n(B)$, such that $H \simeq pB^n$.
\end{theorem}

Twisting the $B$-module structure can be done continuously over a space $U$. We will show that under mild assumptions the resulting spaces can be embedded into a trivial Hilbert $B$-module bundle with fiber $\HA{B}$, so it is possible to untwist the right $B$-module structure. In case the fibers are of $B$-finite rank, we will see that the resulting bundle is in fact locally trivial. The setup is as follows: Let $U$ be a topological space and $\sigma \colon U \to \Aut{B}$ be a continuous map. Let $H$ be a countably generated right Hilbert $B$-module and let $\trivial{\rtw{H}{\sigma}} = U \times H$ be the bundle of right Hilbert $B$-modules over $U$, where the fiber $H_x$ at $x \in U$ carries the right Hilbert $B$-module structure of $\rtw{H}{\sigma(x)}$. Since $\sigma(x)$ preserves the inner product for every $x \in U$, the space $\trivial{\rtw{H}{\sigma}}$ carries the topology of a trivial bundle of Banach spaces over $U$. 
\begin{lemma} \label{lem:emb}
Let $U$, $\sigma$ and $H$ be as above, then there is an embedding
\[
	\Psi \colon \trivial{\rtw{H}{\sigma}} \to \trivial{\HA{B}}
\]
which is a homeomorphism onto a subspace $\mathcal{K} \subset \trivial{\HA{B}}$, preserves the fibers and is a morphism of Hilbert $B$-modules on each of them. In case $H$ is of $B$-finite rank, $\mathbb{H}$ may be chosen to be finite-dimensional.
\end{lemma}
\begin{proof}
Let $\mathbb{H}$ be a separable Hilbert space of countable dimension. By Kasparov's stabilization theorem, $H$ is isomorphic to a direct summand in $\HA{B}$, i.e.\ there is a projection $p_0 \in \bddEnd{}{B}{\HA{B}} = M(\cptEnd{}{B}{\HA{B}}) = M(\K \otimes B)$ in the stable multiplier algebra of $B$ with $H \simeq p_0(\HA{B})$ as right Hilbert $B$-modules. By theorem \ref{thm:AfinRank} we may choose $\mathbb{H}$ to be finite dimensional if $H$ is of $B$-finite rank.

Every element $\tau \in \Aut{B}$ is continuous as a map $B \to B$, when $B$ is equipped with the strict topology, therefore it extends to an automorphism ${\rm Ad}_{\tau} \in M(B)$. If the latter is identified with $\bddEnd{}{B}{B}$, then ${\rm Ad}_{\tau}(T) = \tau \circ T \circ \tau^{-1}$, hence the notation. Now let $\sigma \colon U \to \Aut{B}$ be continuous with respect to the pointwise norm topology on $\Aut{B}$. Denote by $\F$ the finite rank operators on $\mathbb{H}$ and let $\F \otimes_{\rm alg} B$ be the algebraic tensor product, which is dense in the $C^*$-algebraic tensor product $\K \otimes B$. We claim that $\id{\K} \otimes \sigma \colon U \to \Aut{\K \otimes B}$ is still continuous for the pointwise norm topology. Let $\varepsilon > 0$ and $A \in \K \otimes B$. Let $f_i \in \F$ be rank $1$-operators and $b_i \in B$ be such that $F = \sum_{i=1}^C f_i \otimes b_i \in \F \otimes_{\rm alg} B$ satisfies $\lVert A - F \rVert < \frac \varepsilon 4$. Then we have
\begin{align*}
	& \lVert (\id{\K} \otimes \sigma)(x_n)(A) -  (\id{\K} \otimes \sigma)(x)(A) \rVert \\
\leq \quad & 	\lVert (\id{\K} \otimes \sigma)(x_n)(A - F) \rVert  + \lVert  (\id{\K} \otimes \sigma)(x)(F- A) \rVert + \sum_{i=1}^C \lVert f_i \otimes (\sigma(x_n)- \sigma(x))(b_i) \rVert \\
\leq \quad & \frac{\varepsilon}{2} + \sum_{i=1}^C \lVert f_i \rVert \cdot \lVert(\sigma(x_n)- \sigma(x))(b_i) \rVert
\end{align*}
By the pointwise norm continuity of $\sigma$, we may choose an $N \in \N$ such that the last summand is smaller than $\frac{\varepsilon}{2}$ for all $n > N$ proving the claim. Let $T \in \bddEnd{}{B}{\HA{B}}$. By combining the two extensions we end up with 
\[
	{\rm Ad}_{\id{\K} \otimes \sigma}(T) \colon U \to M(\K \otimes B) = \bddEnd{}{B}{\HA{B}} \quad ; \quad x \mapsto {\rm Ad}_{\id{\K} \otimes \sigma(x)}(T)
\]
which is continuous if the target space is equipped with the strict topology. Slightly abusing the notation, we will denote this map by ${\rm Ad}_{\sigma}(T)$. Likewise, every $\tau \in \Aut{B}$ induces a Banach space isomorphism of $\HA{B}$ onto itself via
\[
	\widetilde{\tau} \colon \HA{B} \to \HA{B} \quad ; \quad v \otimes b \mapsto v \otimes \tau(b) \ .
\]
Denote by $\mathcal{L}(\HA{B})$ the bounded linear maps on the Banach space $\HA{B}$. A calculation similar to the one given above shows that the map induced by $\widetilde{\sigma}$, i.e. $\widetilde{\sigma} \colon U \to \mathcal{L}(\HA{B})$ is strongly continuous. As a map between Hilbert $B$-modules $\widetilde{\sigma}(x)$ maps $\HA{B}$ to $(\HA{B})_{\sigma(x)}$.

Let $p_x = {\rm Ad}_{\sigma(x)^{-1}}(p_0)$. Let $\iota_H \colon H \to \HA{B}$ be the injection onto $p_0(\HA{B})$ and set 
\begin{equation} \label{eqn:K}
	\mathcal{K} = \left\{ (x,v) \in U \times (\HA{B}) \ \mid \ p_x\,v = v \right\}\ .
\end{equation}
There is a map of spaces fibered in Hilbert $B$-modules, which is fiberwise a Hilbert $B$-module isomorphism. It is given by
\[
	\Psi \colon \trivial{\rtw{H}{\sigma}} \to \mathcal{K} \quad ; \quad (x,v) \mapsto \left(x,\widetilde{\sigma}(x)^{-1}\left(\iota_H(v)\right)\right)
\]
with inverse
\[
	\Phi \colon \mathcal{K} \to \trivial{\rtw{H}{\sigma}} \quad ; \quad (x,v) \mapsto \left(x,\iota_H^{-1}\left(\widetilde{\sigma}(x)(v)\right)\right)\ .
\]
Indeed, we have $p_x\,\widetilde{\sigma}(x)^{-1}(\iota_H(v)) = \widetilde{\sigma}(x)^{-1}(p_0\,\iota_H(v))$ for $v \in \rtw{H}{\sigma(x)}$, therefore $\Psi$ is well-defined. Since $p_0(\widetilde{\sigma}(x)(v)) = \widetilde{\sigma}(x)(p_x v) = \widetilde{\sigma}(x)(v)$ for $(x,v) \in \mathcal{K}_x$, $\Phi$ is well-defined. Using the strong continuity of $x \mapsto \widetilde{\sigma}(x)^{-1}$, it is easy to see that $\Phi$ and $\Psi$ are continuous. It is clear that they are inverse to each other and a small calculation shows that they preserve the right $B$-module structure and the inner products. Therefore we have embedded $\trivial{\rtw{H}{\sigma}}$ as a subspace $\mathcal{K}$ into the trivial right Hilbert $B$-module bundle $\trivial{\HA{B}}$ over~$U$.
\end{proof}

\begin{lemma}\label{lem:loctriv}
Let $U$, $H$ and $\sigma$ be as in lemma \ref{lem:emb} and suppose additionally that $H$ is of $B$-finite rank. Then each point $y \in U$ has a neighborhood $y \in V \subset U$ such that there exists a homeomorphism, which is fiberwise an isomorphism of right Hilbert $B$-modules 
\[
	\kappa_V \colon V \times \rtw{H}{\sigma(y)} \to \trivial{\rtw{H}{\sigma}}
\]
trivializing the bundle around $y$.
\end{lemma}

\begin{proof}\label{pf:loctriv}
Fix $y \in U$ and let $\mathcal{K}$, $p_0$ and $p_x$ be defined as in lemma \ref{lem:emb} (see (\ref{eqn:K})) with $\mathbb{H}$ finite dimensional. Since we can identify $\trivial{\rtw{H}{\sigma}}$ with its image $\mathcal{K} \subset U \times (\HA{B})$ it suffices to show that $\mathcal{K} \to U$ is locally trivial. Let $n = \dim{\mathbb{H}}$ and note that $p_0 \in M_n(B) = \cptEnd{}{B}{\HA{B}}$, therefore
\[
	\Lambda \colon U \to \bddEnd{}{B}{\HA{B}} \quad ; \quad x \mapsto p_x = {\rm Ad}_{\sigma(x)^{-1}}(p_0)\ .
\]
is not just strictly continuous, but also norm-continuous. Thus, there is a small neighborhood $V \subset U$ of $y$, such that $\lVert p_x - p_y \rVert < 1$ for all $x \in V$. It follows that $p_x$ and $p_y$ are unitarily equivalent via $u_x \in M(M_n(B))$, which is constructed as follows: Define $v_x = 1 - p_x - p_y + 2p_xp_y \in M(M_n(B))$, which is norm continuous. Then we have $v_xv_x^* = v_x^*v_x = 1 - (p_x-p_y)^2$, which is invertible and also norm continuous as a function in $x$. Therefore $u_x = v_x\lvert v_x \rvert^{-1}$ is a unitary, which satisfies $p_xu_x = u_x\,p_y$. All operations applied to $v_x$ this way are norm continuous, therefore $x \mapsto u_x$ is a fortiori strictly continuous. The local trivialization of $\mathcal{K}$ over $V$ is now given by 
\[
	\widetilde{\kappa}_V \colon V \times \mathcal{K}_y \to \mathcal{K} \quad ; \quad (x,v) \mapsto (x,u_xv)
\]
with inverse induced by $u_x^*$.
\end{proof}

The previous lemma worked because of the following rigidity result, which we state as a separate corollary, since it may be of interest in itself.
\begin{corollary}\label{cor:projSimilar}
Let $A$ be a $C^*$-algebra. Suppose $U$ is a path-connected space and $\sigma \colon U \to \Aut{A}$ is a continuous map. Let $t \in M_n(A)$ be a projection. Then for any two points $x,y \in U$ the projections $\sigma(x)(t)$ and $\sigma(y)(t)$ are unitarily equivalent.
\end{corollary}

\begin{proof}\label{pf:projSimilar}
Choose a path in $U$ connecting $x$ and $y$ parametrized by $I = [0,1]$. It suffices to see that 
\[
	J = \left\{ z \in I\ \mid \ \sigma(z)(t) \sim \sigma(x)(t) \right\}
\] 
is all of $I$, where $\sim$ denotes unitary equivalence. But arguing just like in lemma \ref{lem:loctriv} we see that the set $J$ is open and closed in $I$.
\end{proof}

\begin{corollary}\label{cor:Normal}
Let $A$ be a unital $C^*$-algebra and suppose $\Aut{A}$ is path-connected in the point-norm topology. Then the image of $\Aut{A}$ in $\Pic{A}$ is a normal subgroup. 
\end{corollary}

\begin{proof}\label{pf:Normal}
Let $H$ be an $A$-$A$-Morita equivalence and let $\sigma \in \Aut{A}$. We need to show that $H \otimes_A \rtw{A}{\sigma} \otimes_A H^* \simeq \rtw{H}{\sigma} \otimes_A H^*$ is isomorphic as an $A$-$A$-bimodule to $\rtw{A}{\tau}$ for some automorphism $\tau \in \Aut{A}$. Since $H$ is finitely generated and projective as a right Hilbert $A$-module, there is a projection $t \in M_n(A)$ and an isomorphism $H \simeq tA^n$ as right Hilbert $A$-modules. The left multiplication induces a $C^*$-algebra isomorphism $\psi \colon A \to t\,M_n(A)\,t$, such that we can identify $H$ as a bimodule with $tA^n$, where the left multiplication is given by $\psi$. The left inner product then corresponds to 
\[
	\scalL{A}{v}{w} = \psi^{-1}(v\,\scalR{w}{\,\cdot\,}{A})
\]
where the brackets on the right hand side denote the right inner product of $tA^n$ and $v\,\scalR{w}{\,\cdot\,}{A} \in t\,M_n(A)\,t$ is the rank $1$-operator induced by $v$ and $w$. Since we assumed $\Aut{A}$ to be path connected, we get a unitary $u \in M_n(A)$ with ${\rm Ad}_u \circ \sigma^{-1}(t) = t$ by co\-rol\-la\-ry~\ref{cor:projSimilar}. Note that 
\[
\tau \colon A \to A \quad ; \quad a \mapsto (\psi^{-1} \circ {\rm Ad}_u \circ \sigma^{-1} \circ \psi) (a)
\]
defines an automorphism of $A$. Observe that
\[
\Theta \colon \rtw{H}{\sigma} \to \ltw{H}{\tau}\quad ; \quad v \mapsto u\,\sigma^{-1}(v)\ .
\]
is well-defined as an $A$-$A$-bimodule morphism since 
\[
\Theta(a \cdot v) = u\,\sigma^{-1}(a \cdot v) = u\,\sigma^{-1}(\psi(a)v) = ({\rm Ad}_u \circ \sigma^{-1} \circ \psi)\,\Theta(v) = \tau(a) \cdot \Theta(v)\ .
\]
It also preserves the inner products, since
\begin{align*}
\scalR{u\,\sigma^{-1}(v)}{u\,\sigma^{-1}(w)}{A}^{\tau} & = \sigma^{-1}\left(\scalR{v}{w}{A}\right) = \scalR{v}{w}{A}^{\sigma} \\
\scalL{A}{u\,\sigma^{-1}(v)}{u\,\sigma^{-1}(w)}^{\tau} & = (\tau^{-1}\circ\psi^{-1}) \left( u\,\sigma^{-1}(v)\,\scalR{u\,\sigma^{-1}(w)}{\,\cdot\,}{A}\right) \\
& = (\tau^{-1}\circ\psi^{-1}) \left( {\rm Ad}_u \circ \sigma^{-1} \left(v\,\scalR{w}{\,\cdot\,}{A}\right)\right) \\
& = \psi^{-1}\left(v\,\scalR{w}{\,\cdot\,}{A}\right) = \scalL{A}{v}{w}^{\sigma}
\end{align*}
Therefore it yields an isomorphism of Morita equivalences and we have 
\[
\rtw{H}{\sigma} \otimes_A H^* \simeq \ltw{H}{\tau} \otimes_A H^* \simeq \ltw{A}{\tau} \simeq \rtw{A}{\tau^{-1}} 
\]
proving the claim.
\end{proof}

\begin{remark}\label{rem:Bfin}
To be of $B$-finite rank is a sufficient condition for $\trivial{\rtw{H}{\sigma}}$ to be locally trivial as we have seen above. But, it is not necessary: For example, if $H = \HA{B}$, then $\trivial{\rtw{H}{\sigma}}$ is trivial as a bundle of right Hilbert $B$-modules via $(x,v \otimes b) \mapsto (x, v \otimes \sigma(x)^{-1}(b))$.	This holds in particular for $\mathbb{H} = \C$.

However, remark \ref{rem:carefully} suggests that in case $B$ is not of continuous trace there may be a map $\sigma \colon U \to \Inn{B}$ such that there is a point $x \in U$, for which no local lift of $\sigma$ to $U(B)$ exists. If $p \in U(M_n(B))$ is a projection on which $\sigma(x)$ acts non-trivially, then for $H = B^n$ the bundle $\trivial{\rtw{H}{\sigma}}$ is \emph{not} locally trivial.
\end{remark}

Let $B$ and $C$ be $C^*$-algebras and let $H$ be a right Hilbert $B$-module and $K$ be a right Hilbert $C$-module. If we have a $*$-homomorphism $\varphi \colon B \to \bddEnd{}{C}{K}$, we can form the interior tensor product with respect to $\varphi$ similar to definition \ref{def:tensorproduct}, i.e.\ $H \otimes_{\varphi} K$ is the algebraic tensor product modulo the relation $hb \otimes k = h \otimes \varphi(b)k$, completed with respect to the norm induced by the inner product
\[
	\scalR{h_1 \otimes k_1}{h_2 \otimes k_2}{B} = \scalR{k_1}{\varphi(\scalR{h_1}{h_2}{A})k_2}{B}\ .
\]
If $\Hb \to U$ is a locally trivial bundle of right Hilbert $B$-modules with trivial fiber $H$ and structure group $\U{}{B}{H}$, $K$ is a right Hilbert $C$-module as above and $\psi \colon B \to \cptEnd{}{C}{K}$ is a $*$-homomorphism, we can form the tensor product $\Hb \otimes_{\psi} K$, which as a set is just 
\begin{equation} \label{eqn:tensorK}
	\Hb \otimes_{\psi} K = \coprod_{x \in X} \{x\} \times \left(\Hb_x \otimes_{\psi} K\right)\ .
\end{equation}
The topology we impose on this space is chosen in such a way that local trivializations are homeomorphisms. This is achieved as follows: Cover $U$ by open sets $U_i$ over which $\Hb$ is trivializable and choose trivializations $\kappa_i \colon \left.\Hb\right|_{U_i} \to U_i \times H$. Then $\kappa_i \otimes \id{K} \colon \left.\Hb \otimes_{\psi} K\right|_{U_i} \to U_i \times (H \otimes_{\psi} K)$ trivializes the above bundle over $U_i$, such that we can pull back the topology of $U_i \times H \otimes_{\psi} K$ to the subbundle $\left.\Hb \otimes_{\psi} K\right|_{U_i}$. Let $c_{ij} = \kappa_j \circ \kappa_i^{-1}$ be the $\U{}{B}{H}$-cocycle associated to $\Hb$, which we interpret as a map $U_{ij} \to \U{}{B}{H}$. To see that the local topologies patch together to form a topology on $\Hb \otimes_{\psi} K$ we only have to check that the induced cocycle $c_{ij} \otimes \id{K} \colon U_{ij} \to \U{}{C}{H \otimes_{\psi} K}$ is continuous with respect to the strict topology, which can be seen by an argument similar to the one used in the proof of the continuity of $\id{\K} \otimes \sigma$ in lemma \ref{lem:emb}. 

If $H$ is of $B$-finite rank and $\sigma \in \Aut{B}$, then the bundle $\trivial{\rtw{H}{\sigma}}$ is locally trivial (see lemma \ref{lem:loctriv}). Therefore we can define the bundle $\trivial{\rtw{H}{\sigma}} \otimes_{\psi} K$. Suppose $\psi \colon U \times B \to \cptEnd{}{C}{K}$ is a continuous map, which is an isomorphism of $C^*$-algebras at every $x \in U$. Let $\trivial{\rtw{H}{\psi^{-1}}}$ be the right Hilbert $\cptEnd{}{C}{K}$-module, where the right module structure at a point $x \in U$ is induced by $\psi^{-1}(x)$. Fix a point $x_0 \in U$ and let $\psi_0 \colon U \times B \to \cptEnd{}{C}{K}$ map every $x \in U$ to the isomorphism at $x_0$. Then we have 
\[
	\trivial{\rtw{H}{\psi^{-1}}} = \left(\trivial{\rtw{H}{\psi^{-1}\circ \psi_0}}\right)_{\psi_0^{-1}} \ , 
\]
where $\psi^{-1}\circ \psi_0 \colon U \to \Aut{B}$, therefore $\trivial{\rtw{H}{\psi^{-1}}}$ is locally trivial, which allows us to define 
\begin{equation}\label{eqn:semitrivial}
	\trivial{H \otimes_{\psi} K} := \trivial{\rtw{H}{\psi^{-1}}} \otimes_{\cptEnd{}{C}{K}} K\ ,
\end{equation}
with respect to the identity map on $\cptEnd{}{C}{K}$. The notation of this locally trivial Hilbert $C$-module bundle stems from the fact that the fiber at $x \in U$ is equal to $H \otimes_{\psi(x)} K$ as a right Hilbert $C$-module.

In a similar fashion as we did above, we can impose a topology on the bundle $\cptEnd{}{B}{\Hb}$, whose fiber over $x \in U$ consists of the compact adjointable $B$-linear endomorphisms of $\Hb_x$. Note that this bundle is trivial in the case of $\trivial{\rtw{H}{\sigma}}$, i.e.\ $\cptEnd{}{B}{\trivial{\rtw{H}{\sigma}}} = \trivial{\cptEnd{}{B}{H}}$. Since we assumed $\psi$ to be an isomorphism at each point we also get 
\begin{equation} \label{eqn:cpts}
	\cptEnd{}{C}{\trivial{H \otimes_{\psi} K}} \simeq \cptEnd{}{B}{\trivial{\rtw{H}{\psi^{-1}}}} = \trivial{\cptEnd{}{B}{H}}\ ,
\end{equation}
where the first isomorphism is the inverse of $T \mapsto T \otimes_{\psi} \id{K}$.

\begin{definition}\label{def:rank1bibundle}
Let $X$ be a topological space. A continuous, locally trivial fiber bundle $\pi \colon \Hb \to X$ is called an ($A$-$B$-)\emph{bimodule bundle} if each fiber $\Hb_x = \pi^{-1}(x)$ carries the structure of a Hilbert $A$-$B$-bimodule isomorphic to some typical fiber $H$ and the structure group reduces to $\biU{A}{B}{H}$. $\Hb$ will be called \emph{Morita bundle} if all fibers are Morita equivalences.
\end{definition}

Since we want the bimodule structure to vary over $X$, which yields a bundle with fibers that are \emph{not} isomorphic as bimodules, we need a weaker notion of Morita bundle. For the rest of this chapter we will assume $A$ and $B$ to be \emph{unital} $C^*$-algebras.

\begin{definition}\label{def:twistedMorita}
Let $X$ be a connected topological space. A con\-ti\-nuous, locally trivial fiber bundle $\pi \colon \Hb \to X$ is called a \emph{twisted $A$-$B$-Morita bundle}, if it is a left Hilbert $A$-module bundle with full typical fiber $H$ and structure group $\U{A}{}{H}$, together with an isomorphism
\begin{equation} \label{eqn:triv}
	X \times B^{\rm op} \to \cptEnd{A}{}{\Hb}
\end{equation}
of algebra bundles, which turns each fiber into an $A$-$B$-Morita equivalence. An isomorphism of twisted $A$-$B$-Morita bundles is an isometric isomorphism $\Psi \colon \Hb \to \Hb'$ of left Hilbert $A$-module bundles intertwining the corresponding algebra bundle maps:
\[
  \begin{xy}
   \xymatrix{
	X \times B^{\rm op} \ar[r] \ar@{=}[d] & \cptEnd{A}{}{\Hb} \ar[d]^{{\rm Ad}_{\Psi}} \\
	X \times B^{\rm op} \ar[r] & \cptEnd{A}{}{\Hb'}
   }
  \end{xy}
\]
If $X$ is not connected, a twisted $A$-$B$-Morita bundle $\Hb \to X$ is a twisted Morita bundle over each of its connected components, where the typical fiber may change from one component to the other.
\end{definition}

\begin{lemma}\label{lem:symmetricDef}
Let $\Hb \to X$ be a twisted $A$-$B$-Morita bundle, then it is also a right Hilbert $B$-module bundle in a canonical way together with an isomorphism
\[
	X \times A \to \cptEnd{}{B}{\Hb}
\]
of bundles of algebras. 
\end{lemma}
\begin{proof}\label{pf:symmetricDef}
Since every left Hilbert $A$-module $H$ is in a canonical way a Hilbert $A$-$\cptEnd{A}{}{H}$-bimodule and all fibers of $\cptEnd{A}{}{\Hb}$ are identified with $B$ via (\ref{eqn:triv}), the fibers of $\Hb$ are indeed also full right Hilbert $B$-modules in a canonical way. In particular, the typical fiber $H$ is a right Hilbert $B$-module. For if we fix a point $x_0 \in M$ and an isomorphism $\Hb_{x_0} \simeq H$, $H$~inherits the Hilbert $B$-module structure of the fiber. Now we \emph{fix} a right Hilbert $B$-module structure on $H$ together with an isomorphism $\cptEnd{A}{}{H} \simeq B$. We need to prove that $\Hb$ is locally trivial as a right Hilbert $B$-module bundle. So, fix a point $y \in M$ and a neighborhood $U \ni y$ allowing a trivialization $\kappa_U \colon U \times H \to \left.\Hb\right|_U$ as a left Hilbert $A$-module bundle. This induces an isomorphism of the trivial algebra bundle over $U$ via 
\[
 B \times U \to \cptEnd{A}{}{\left.\Hb\right|_U} \to \cptEnd{A}{}{H} \times U \simeq B \times U \ ,
\]
which can be seen as a map $\sigma \colon U \to \Aut{B}$, such that 
\[
	\kappa_U \colon U \times H_{\sigma} \to \left.\Hb\right|_U
\]
is fiberwise an isomorphism of bimodules in the notation of lemma \ref{lem:emb}. By lemma \ref{lem:loctriv} (which applies due to remark \ref{rem:fingenproj} since $A$ and $B$ are unital) there exists a neighborhood $V$ of $y$, such that $\left.\Hb\right|_U$ is trivial over $V$ as a right Hilbert $B$-module bundle. If $H$ is a full left Hilbert $A$-module and $C = \cptEnd{A}{}{H}$, then $\cptEnd{}{C}{H} = I_A = A$. Applying this to the fibers yields the trivialization $X \times A \to \cptEnd{}{B}{\Hb}$, i.e.\ it maps $(x,a)$ to the endomorphism $v \mapsto a\cdot v$, where $v$ lives in the fiber $\Hb_x$.
\end{proof}

\begin{corollary}\label{cor:twistedconjugate}
If $\Hb \to X$ is a twisted $A$-$B$-Morita bundle, then the space $\Hb^*$, which we get by taking the conjugate fiberwise, is a twisted $B$-$A$-Morita bundle. 	
\end{corollary}
\begin{proof}\label{pf:twistedconjugate}
By lemma \ref{lem:symmetricDef} we can see $\Hb$ as a right Hilbert $B$-module bundle together with a map $\Psi$ trivializing $\cptEnd{}{B}{\Hb}$ via $A$. Taking the fiberwise conjugate therefore yields a left Hilbert $B$-module bundle. The right $A$-action is given by the map
\[
	X \times A^{\rm op} \to \cptEnd{B}{}{\Hb}\ ,
\]
which sends $(x,a)$ to $\Psi(x,a)^*$.
\end{proof}

\begin{example}\label{ex:twistedMnC}
Let $A = M_n(\C)$. Note that $\Aut{A}= \Inn{A} = PU(n)$ and that up to isometric isomorphism $A = M_n(\C)$ is the only $A$-$A$-Morita equivalence. Therefore a twisted $A$-$A$-Morita bundle $\Hb$ over $X$ should induce an honest $A$-$A$-Morita bundle, which we will now determine. By definition $\Hb$ is a principal $U(n)$-bundle $P \to X$ such that 
\[
	\Hb = P \times_{U(n)} M_n(\C)
\]
where $U(n)$ acts via the canonical \emph{right} action. The associated automorphism bundle $P \times_{\rm Ad} PU(n)$ is trivializable. Thus, from the exact sequence of pointed sets
\[
	H^1(M, U(1)) \to H^1(M, U(n)) \to H^1(M, PU(n)) 
\]  
we see that $P$ actually reduces to a principal $S^1$-bundle, which we continue to denote by $P$. Therefore $\Hb$ has the form
\[
	\Hb = P \times_{U(1)} M_n(\C) = L_1 \otimes M_n(C) \ ,
\]
where $L_1$ is the line bundle associated to $P$, together with a bundle isomorphism
\[
	X \times M_n(\C) \to X \times M_n(\C)\ ,
\]
which we can see as a map $\sigma \colon X \to PU(n)$. Pulling back the principal $U(1)$-bundle $U(n) \to PU(n)$ via $\sigma$ we get another associated line bundle $L_2 \to X$. A simple calculation shows that 
\[
X \times M_n(\C)_{\sigma}\ \simeq\ L_2^* \otimes M_n(\C) 
\] 
as a bundle of Hilbert $A$-$A$-bimodules (where we use the notation of lemma \ref{lem:emb}). Let $L = L_1 \otimes L_2^*$, then we get an isomorphism $\Hb \simeq L \otimes M_n(\C)$ as a bimodule bundle.
\end{example}

\subsection{Twisted tensor products} 
\label{ssub:subsubsection_name}
Now, suppose that $\Hb$ is a twisted $A$-$B$-Morita bundle over $X$ with typical fiber $H$, likewise let $\Nb$ be a twisted $B$-$C$-Morita bundle over the same space with fiber $N$. $\Hb$ is a right Hilbert $B$-module bundle and we have a trivialization $\psi \colon X \times B \to \cptEnd{}{C}{\Nb}$ if we view $\Nb$ as a right Hilbert $C$-module bundle. So we should be able to define a tensor product of twisted Morita bundles: As a set we simply take 
\[
	\Hb \otimes_{B} \Nb = \coprod_{x \in X} \left\{ x\right\} \times \Hb_x \otimes_{\psi(x)} \Nb_x\ .
\]
Let $\pi \colon \Hb \otimes_{B} \Nb \to X$ be the canonical projection. Again, we impose a topology on this space in such a way that local trivializations are homeomorphisms, but this time we use $\trivial{H \otimes_{\psi} N}$ as a model space (see (\ref{eqn:semitrivial})). Cover $X$ by open sets $U_i$, such that $\Hb$ and $\mathcal{K}$ are trivial over $U_i$ via trivializations $\kappa_i \colon \left.\Hb\right|_{U_i} \to U_i \times H$ and $\tau_i \colon \left.\Nb\right|_{U_i} \to U_i \times N$ respectively. Let 
\[
	\psi_i \colon U_i \times B \to \cptEnd{}{C}{N} \quad ; \quad \psi_i(x,b) =  \tau_i(x, \cdot ) \circ \psi(x,b) \circ \tau_i^{-1}(x, \cdot) \in \cptEnd{}{B}{N}
\]
and note that $\kappa_i \otimes \tau_i \colon \left.\Hb \otimes_B \Nb\right|_{U_i} \to \trivial{H \otimes_{\psi_i} N}$ is well defined on the algebraic tensor product, since for $h \in \Hb_x$ and $n \in \Nb_x$ we have
\begin{align*}
	(\kappa_i \otimes \tau_i)(h\,b \otimes n) & = \kappa_i(h)b \otimes \tau_i(n) = \kappa_i(h) \otimes \psi_i(x,b)\tau_i(n) \\
	& = \kappa_i(h) \otimes \tau_i(\psi(x,b)n) = (\kappa_i \otimes \tau_i)(h \otimes \psi(x,b)n)
\end{align*}
and a similar calculation shows that $\kappa_i \otimes \tau_i$ preserves the $C$-valued inner product. Therefore it extends to an isomorphism of right Hilbert $C$-module bundles over $U_i$. We equip $\left.\Hb \otimes_B \Nb\right|_{U_i}$ with the topology induced from $\trivial{H \otimes_{\psi_i} N}$. Let $c_{ij} \colon U_{ij} \to \U{}{B}{H}$ be the cocycle of $\Hb$ obtained from $\kappa_j \circ \kappa_i^{-1}$, likewise denote by $d_{ij} \colon U_{ij} \to \U{}{C}{N}$ the one of $\Nb$. To see that the partial topologies patch together to form a topology on $\Hb \otimes_B \Nb$ we need to see that the following section
\[
	c_{ij} \otimes d_{ij} \colon U_{ij} \to {\rm Iso}(\trivial{H \otimes_{\psi_i} N}, \trivial{H \otimes_{\psi_j} N})
\]
of the bundle of isometric isomorphisms between $\trivial{H \otimes_{\psi_i} N}$ and $\trivial{H \otimes_{\psi_j} N}$ is continuous with respect to the strict topology. Let $(x_k)_{k \in I}$ be a net in $U_{ij}$ converging to $x \in U_{ij}$ and define $\sigma_i \colon U_{i} \to \Aut{B}$ by $\sigma_i(y) = \psi_i(x)^{-1} \circ \psi_i(y)$, where we see $\psi_i$ as a map $U_i \to {\rm Iso}(B, \cptEnd{}{C}{N})$. Then we have 
\[
	\trivial{H \otimes_{\psi_i} N} = \trivial{\rtw{H}{\sigma_i^{-1}}} \otimes_{\psi_i(x)} N\ .
\]
The cocycle $c_{ij}$ can be viewed as a continuous bundle map $c_{ij} \colon \trivial{\rtw{H}{\sigma_i^{-1}}} \to \trivial{\rtw{H}{\sigma_j^{-1}}}$ intertwining $\sigma_i$ and $\sigma_j$, i.e.\ $c_{ij}(h \cdot \sigma_i(b)) = c_{ij}(h)\cdot \sigma_j(b)$ due to the $B$-linearity of $c_{ij}$ (the dot denotes the \emph{twisted} multiplication). Let $V \subset U_{ij}$ be an open neighborhood of $x$ such that $\trivial{\rtw{H}{\sigma_i^{-1}}}$ and $\trivial{\rtw{H}{\sigma_j^{-1}}}$ are trivial over $V$ via
\[
	\varphi_V^i \colon \trivial{\rtw{H}{\sigma_i^{-1}}} \to V \times H\ ,
\]
then we need to check that $(\varphi_V^j \otimes \id{N}) \circ (c_{ij} \otimes d_{ij}) \circ (\varphi_V^i \otimes \id{N})^{-1}(x_n) \in {\rm Iso}(H \otimes_{\psi_i(x)} N, H \otimes_{\psi_j(x)} N)$ converges strictly to its values at $x$ for $n > D$, where $D$ is chosen such that $x_n \in V$. Let $\mathcal{L}(H)$ be the bounded linear operators on the Banach space $H$ and note that $\widetilde{c}_{ij} = \varphi_V^j \circ c_{ij} \circ (\varphi_V^i)^{-1} \colon V \to \mathcal{L}(H)$ is strongly continuous. Thus, we have for $h \otimes k \in H \otimes_{\psi_i(x)} N$
\begin{align*}
  & \lVert \left(\widetilde{c}_{ij} \otimes d_{ij}\right)(x_n)(h \otimes k) - \left(\widetilde{c}_{ij} \otimes d_{ij}\right)(x)(h \otimes k) \rVert \\
\leq\ & \lVert \widetilde{c}_{ij}(x_n)(h) \otimes d_{ij}(x_n)(k) - \widetilde{c}_{ij}(x_n)(h) \otimes d_{ij}(x)(k) \rVert \\ 
+\ & \lVert \widetilde{c}_{ij}(x_n)(h) \otimes d_{ij}(x)(k) - \widetilde{c}_{ij}(x)(h) \otimes d_{ij}(x)(k) \rVert \\
\leq\ & \lVert h \rVert \cdot \lVert d_{ij}(x_n)(k) - d_{ij}(x)(k) \rVert + \lVert \widetilde{c}_{ij}(x_n)(h) - \widetilde{c}_{ij}(x)(h) \rVert \cdot \lVert k \rVert\ ,
\end{align*}
which tends to $0$ for $n$ large enough by the continuity of $\widetilde{c}_{ij}$ and $d_{ij}$. This shows the strict continuity on elementary tensors, so by the triangle inequality also on linear combinations and by the density of the algebraic tensor product for any vector. Since $\left(c_{ij} \otimes d_{ij}\right)^* = \left(c_{ij}^{-1} \otimes d_{ij}^{-1}\right) = c_{ji} \otimes d_{ji}$ we are done.
\begin{remark} \label{rem:twistedModule}
	Note that, if we do not assume $\Hb$ to be a twisted Morita bundle, but only a locally trivial bundle of finitely generated projective right Hilbert $B$-modules, we are still able to define $\Hb \otimes_B \Nb$, which will then be a locally trivial bundle of right Hilbert $C$-modules. Likewise we can form the partial stabilization $(\trivial{\HA{B}}) \otimes_B \Nb$ for a separable Hilbert space $\mathbb{H}$ of infinite dimension, even though $H = \HA{B}$ is not finitely generated as a right Hilbert $B$-module, since in this case $\trivial{\rtw{H}{\sigma}}$ is trivial (see remark \ref{rem:Bfin}). The result will be a twisted Morita $(\K \otimes B)$-$C$-bundle, in particular it is a locally trivial right Hilbert $C$-module bundle. Thus, we can form the full stabilization $(\trivial{\HA{B}}) \otimes_B \Nb \otimes_C (\trivial{\HA{C}})$, which is a twisted Morita $(\K \otimes B)$-$( \K \otimes C)$-bundle.
\end{remark}

\begin{definition}\label{def:MoritaBG}
A \emph{Morita bundle gerbe} for $A$ over a manifold $M$ is a fibration $\pi \colon Y \to M$ together with a twisted $A$-$A$-Morita bundle $\Hb \to Y^{[2]}$ and an isomorphism of twisted Morita bundles
\[
	\mu \colon \pi_{12}^* \Hb \otimes_A \pi_{23}^* \Hb \to \pi_{13}^* \Hb \ ,
\] 
which can be seen as a multiplication map and should satisfy the following associativity condition over $Y^{[3]}$:
\[
  \begin{xy}
   \xymatrix{
	\left(\pi_{12}^* \Hb \otimes_A \pi_{23}^*\Hb\right) \otimes_A \pi_{34}^*\Hb \ar@{=}[rr] \ar[d]^{\mu \otimes \id{}} & & \pi_{12}^* \Hb \otimes_A \left(\pi_{23}^*\Hb \otimes_A \pi_{34}^*\Hb\right) \ar[d]^{\id{} \otimes \mu} \\
	\pi_{13}^* \Hb \otimes_A \pi_{34}^*\Hb \ar[dr]^{\mu} & & \pi_{12}^* \Hb \otimes_A \pi_{24}^*\Hb \ar[dl]^{\mu}\\
	& \pi_{14}^* \Hb & 
   }
  \end{xy}
\]
\end{definition}

\begin{remark}\label{rem:localsection}
Since $Y$ plays the role of an auxiliary space, its most important property is that the surjective projection map allows local sections. Thus, if we want to work in the smooth category, we may change the definition from fibrations to surjective submersions. 
\end{remark}

\begin{example}
Let $\Qb \to Y$ be a twisted $A$-$A$-Morita bundle over $Y$, then the canonical pairing $\Qb^* \otimes_A \Qb \to \trivial{A}$ turns $\Hb = \pi_1^*\Qb \otimes_A \pi_2^*\Qb^*$ into a Morita bundle gerbe. In analogy to the situation for bundle gerbes this will be called the \emph{trivial Morita bundle} associated to $\Qb$.
\end{example}

\begin{example} \label{ex:groupaction}
Let $G$ be a topological group and $\varphi \colon G \to \Aut{A}$ be a homomorphism. Let $P \to M$ be a principal $G$-bundle. Since the action on the fibers of $P$ is free and transitive, there is a well-defined map $g \colon P^{[2]} \to G$ sending $(p_1,p_2)$ to the element $g(p_1,p_2) \in G$ with the property $p_1 \cdot g(p_1,p_2) = p_2$. Note that $g(p_1,p_2) \cdot g(p_2, p_3) = g(p_1,p_3)$ Therefore 
\[
	\Hb = P^{[2]} \times \Atw{\varphi \circ g}\ ,
\]
where the fiber over $(p_1,p_2)$ is the Hilbert $A$-$A$-bimodule $\Atw{\varphi(g(p_1,p_2))}$ is a Morita bundle via the multiplication map from example \ref{ex:autTwist}
\[
	\mu \colon \Atw{\varphi(g(p_1,p_2))} \otimes_A \Atw{\varphi(g(p_1,p_2))} \to \Atw{\varphi(g(p_1,p_2)g(p_2,p_3))} = \Atw{\varphi(g(p_1,p_3))} \ .
\]
\end{example}

\begin{example} \label{ex:groupoidaction}
As a higher algebraic extension of the last example, consider the following: Let $Y \to M$ be a fibration over a manifold $M$. The fiber product $Y^{[2]}$ defines a groupoid with object space $Y$ and the two projections as the source and range map. Suppose we have a weak action in the sense of \cite{paper:HigherCats} of $Y^{[2]}$ on $A$, i.e.\ continuous maps $\lambda \colon Y^{[2]} \to \Aut{A}$ and $g \colon Y^{[3]} \to U(A)$, such that 
\begin{align*}
	\lambda(y_1, y_2) \circ \lambda(y_2,y_3) & = {\rm Ad}_{g(y_1,y_2,y_3)} \circ \lambda(y_1,y_3)\ , \\
	\lambda(y_1,y_2)(g(y_2,y_3,y_4)) \cdot g(y_1,y_2,y_4) & = g(y_1,y_2,y_3) \cdot g(y_1,y_3,y_4)\ .
\end{align*}
Then $\Hb = Y^{[2]} \times A_{\lambda}$ is a Morita bundle gerbe, where the multiplication map $\mu$ sends $A_{\lambda(y_1,y_2)} \otimes_A A_{\lambda(y_2,y_3)}$ to $A_{\lambda(y_1,y_2) \circ \lambda(y_2,y_3)}$ via the canonical isomorphism and then uses right multiplication by $g(y_1,y_2,y_3)$ to map into $A_{\lambda(y_1,y_3)}$. The condition on $g$ ensures the associativity of $\mu$.

A special case of this arises from Busby-Smith twisted actions \cite{paper:BusbySmith}. If a pair $\alpha \colon G \to \Aut{A}$ and $u \colon G \times G \to U(A)$ yields such an action and $P$ is a principal $G$-bundle over $M$ with $\varphi \colon P^{[2]} \to G$ being the canonical map, then we can set $\lambda = \alpha \circ \varphi$ and $g = u \circ (\pi_{12}^*\varphi, \pi_{23}^*\varphi)$ to get a Morita bundle gerbe.
\end{example}

\begin{example}{\label{ex:matrixMbgs}}
Let $A = M_n(\C)$ be a matrix algebra. As we have seen in example \ref{ex:twistedMnC} every twisted Morita bundle $\Hb \to Y^{[2]}$ takes the form $L \otimes M_n(\C) \to Y^{[2]}$ for a line bundle $L$. Using the canonical Morita equivalence $\C^n$ between $M_n(\C)$ and $\C$ we can reduce $\Hb$ to the line bundle $L$ and the multiplication $\mu$ to a line bundle isomorphism, which turns $L$ into an $S^1$-bundle gerbe.
\end{example}

\begin{remark}\label{rem:fiberIsA}
If $\Hb \to Y^{[2]}$ is a Morita bundle gerbe and $(y_1, y_2) \in Y^{[2]}$, then $\mu$ induces an isomorphism on the fibers:
\[
	\Hb_{(y_1,y_1)} \otimes_A \Hb_{(y_1,y_2)} \simeq \Hb_{(y_1,y_2)} \quad \Rightarrow \quad \Hb_{(y_1,y_1)} \simeq \Hb_{(y_1,y_2)} \otimes_A \Hb_{(y_1,y_2)}^* \simeq A\ .
\]
So, the diagonal is isomorphic to the trivial Hilbert $A$-$A$-bimodule $A$, but this means that the typical fiber of $\Hb$ over the connected component of the diagonal in $Y^{[2]}$ is isomorphic to $A$ as a left (or equivalently right) Hilbert $A$-module, but of course not as a bimodule. Moreover, we get an isomorphism
\[
	\Hb_{(y_1,y_2)} \otimes_A \Hb_{(y_2,y_1)} \simeq \Hb_{(y_1,y_1)} \simeq A \quad \Rightarrow \quad \Hb_{(y_1,y_2)} \simeq \Hb_{(y_2,y_1)}^*\ .
\]
Let $\Delta \colon Y \to Y^{[2]}$ be the embedding onto the diagonal in the fiber product, then we set $\Delta \Hb = \Delta^*\Hb$. If $\Hbop$ is the pullback of $\Hb$ via the switch map $Y^{[2]} \to Y^{[2]}$ sending $(y_1,y_2)$ to $(y_2,y_1)$, then $\Hb^*\ \simeq\ \Hbop \otimes_A \pi_Y^*\Delta\Hb\ \simeq\ \Hbop$. Note that the isomorphism $\Delta\Hb \simeq \trivial{A}$ has to be \emph{chosen}. We have $\Hb \otimes \Delta\Hb \simeq \Hb$ via $\mu$ and $\Hb \otimes \trivial{A} \to \Hb$ via right multiplication. But it is a priori not clear that $\Delta\Hb \to \trivial{A}$ respects this structure. A Morita bundle gerbe $\Hb$ together with an isomorphism, which respects the product, will be called \emph{unital}. 
\end{remark}

\begin{definition}\label{def:stableequiv}
Let $\Hb \to Y^{[2]}$ and $\Hb' \to Y'^{[2]}$ be two Morita bundle gerbes with respect to $A$. Let $\pi_{Y^{[2]}} \colon (Y \times_M Y')^{[2]} \to Y^{[2]}$, $\pi_{Y'^{[2]}} \colon (Y \times_M Y')^{[2]} \to Y'^{[2]}$ be the canonical projections and let $\pi_i \colon (Y \times_M Y')^{[2]} \to Y \times_M Y'$ be the projection to the $i$th factor. If there exists a twisted $A$-$A$-Morita bundle $\Qb \to Y \times_M Y'$, such that 
\[
	\pi_{Y^{[2]}}^*\Hb \ \simeq \ \pi_1^*\Qb \otimes_A \pi_{Y'^{[2]}}^*\Hb' \otimes_A \pi_2^*\Qb^*\ ,
\]
then $\Hb$ and $\Hb'$ will be called \emph{stably equivalent}. As we will see, this is an equivalence relation, which is weaker than isomorphism of Morita bundle gerbes. The stable equivalence classes of twisted Morita bundles gerbes for $A$ over $M$ will be denoted by $\Tw{M}{A}$. This is a pointed set with the class of the trivial Morita bundle gerbe $M = M^{[2]} \times A \to M$ as distinguished element. 
\end{definition}

Of course, we need to show that the above definition indeed yields an equivalence relation. For the proof of transitivity and in many places later we will need the next lemma, which is essentially a folklore result. We include the easy proof for the convenience of the reader, since the lemma will play quite a fundamental role in the following.  

\begin{lemma}\label{lem:descentlemma}
Let $E \to Y$ be a locally trivial bundle with fiber $V$ over the total space of a fibration $\pi \colon Y \to M$. Assume that there is a bundle isomorphism
\[
\phi \colon \pi_2^* E \overset{\simeq}{\longrightarrow} \pi_1^*E \ ,
\]
where $\pi_i \colon Y^{[2]} \to Y$ denote the canonical projections. If the following associativity diagram over $Y^{[3]}$ commutes, 
\[
  \begin{xy}
   \xymatrix{
	\pi_3^* E \ar[rr]^{\pi_{23}^*\phi} \ar[drr]_{\pi_{13}^*\phi} & & \pi_2^* E \ar[d]^{\pi_{12}^*\phi} \\
	& & \pi_1^* E
   }
  \end{xy}
\]
then there is a bundle $\widetilde{E} \to M$ with fiber $V$ and an isomorphism $E \to \pi^*\widetilde{E}$. This construction is functorial in the sense that a bundle endomorphism that commutes with $\phi$ induces an endomorphism of $\widetilde{E}$.
\end{lemma}

\begin{proof}
Cover $M$ by contractible sets $M \subset \bigcup_{i \in I} U_i$ and choose sections $\sigma_i \colon U_i \to Y$. Let $E_i = \sigma_i^* E$ and denote the maps induced by $\phi$ on the double intersections $U_{ij} = U_i \cap U_j$ by $\phi_{ij} \colon E_j \to E_i$. Over the non-empty triple intersections $U_{ijk}= U_i \cap U_j \cap U_k$ the associativity condition translates into the commutative diagram
\[
  \begin{xy}
   \xymatrix{
	E_k \ar[rr]^{\phi_{jk}} \ar[drr]_{\phi_{ik}} & & E_j \ar[d]^{\phi_{ij}} \\
	& & E_i
   }
  \end{xy}
\]
In particular, setting $k = j$ we see that $\phi_{jj} = \id{E_j}$, from which we deduce $\phi_{ij}^{-1} = \phi_{ji}$. It follows that there is an equivalence relation over $U_{ij}$ generated by $(j,v) \sim (i, \phi_{ij}(v))$ for $(j,v) \in E_j$ on the space $\coprod_{i \in I} E_i$. Let 
\[
\widetilde{E} = \coprod_{i \in I} E_i \ /\ \sim
\]
be the quotient by this relation. There is a canonical projection $\widetilde{E} \to M$ turning this space into a locally trivial bundle with fiber $V$ over $M$. Let $\phi_{(y_2, y_1)} \colon E_{y_1} \to E_{y_2}$ be the isomorphism on the fibers of $E$ induced by $\phi$, then for $[i,v] \in \widetilde{E}$ and $y \in Y$ we define
\[
	\pi^*\widetilde{E} \to E \quad ; \quad (y, [i,v]) \mapsto \phi_{(y,\,\sigma_i(\pi(y)))}(v)\ ,
\]
which is well-defined, since for $\pi(y) \in U_{ij}$ we have 
\[ 
	\phi_{(y,\,\sigma_j(\pi(y)))}(v) = \phi_{(y,\,\sigma_i(\pi(y)))}( \phi_{ij}(v) ) \ .
\]
It is an isomorphism with inverse
\[
	E \to \pi^*\widetilde{E}  \quad ; \quad v \mapsto (y, [i, \phi_{(\sigma_i(\pi(y)), y)}(v)])
\]
for $v \in E_y$.
\end{proof}

\begin{remark}\label{rem:intertwiner}
Suppose $\Qb \to Y' \times_M Y$ induces a stable equivalence between two Morita bundle gerbes $\Hb \to Y^{[2]}$ and $\Hb' \to Y'^{[2]}$, i.e.\ there is an isomorphism
\[
	\pi_{Y'^{[2]}}^* \Hb' \ \simeq \ \pi_1^* \Qb \otimes_A \pi_{Y^{[2]}}^* \Hb \otimes_A \pi_2^*\Qb^* 
\]
On the fibers, this yields for $((y_1',y_2'), (y_1,y_2)) \in Y'^{[2]} \times_M Y^{[2]}$
\begin{equation} \label{eqn:stableiso}
	\Hb'_{(y_1',y_2')}\ \simeq\ \Qb_{(y_1',y_1)} \otimes_A \Hb_{(y_1,y_2)} \otimes_A \Qb_{(y_2',y_2)}^* \ ,
\end{equation}
which induces an isomorphism $\Qb_{(y_1',y_1)} \otimes_A \Hb_{(y_1,y_2)} \to \Hb'_{(y_1',y_2')}\otimes_A \Qb_{(y_2',y_2)}$. The compatibility with the multiplication maps $\mu$ and $\mu'$ translates into the following commutative diagram
\[
  \begin{xy}
   \xymatrix{
	\Hb'_{(y_1',y_2')} \otimes_A \Qb_{(y_2',y_1)} \otimes_A \Hb_{(y_1,y_2)} \ar[rr] \ar[d] & & \Hb'_{(y_1',y_2')} \otimes_A \Hb'_{(y_2',y_3')} \otimes_A \Qb_{(y_3',y_2)} \ar[d]^{\mu' \otimes \id{\Qb}} \\ 
	\Qb_{(y_1',y_3)} \otimes_A \Hb_{(y_3,y_1)} \otimes_A \Hb_{(y_1,y_2)} \ar[dr]^{\id{\Qb} \otimes \mu} & & \Hb'_{(y_1',y_3')} \otimes_A \Qb_{(y_3',y_2)} \ar[dl] \\
	& \Qb_{(y_1',y_3)} \otimes_A \Hb_{(y_3,y_2)}
   }
  \end{xy}
\]
Therefore $\Qb$ may be thought of as an intertwiner between $\Hb$ and $\Hb'$, which is compatible with the multiplicative structures.
\end{remark}

\begin{corollary}\label{cor:twistedDescent}
Let $Y \to M$ be as in the last lemma and suppose $\Qb \to Y$ is a twisted $A$-$B$-Morita bundle over $Y$. If there is an isomorphism $\phi \colon \pi_2^*\Qb \to \pi_1^*\Qb$ of twisted Morita bundles, which satisfies the conditions of lemma \ref{lem:descentlemma}, then there is a twisted $A$-$B$-Morita bundle $\widetilde{\Qb}$ over $M$, such that $\Qb \simeq \pi^*\widetilde{\Qb}$.
\end{corollary}

\begin{proof}\label{pf:twistedDescent}
By the last lemma, $\Qb$ descends to a bundle $\widetilde{\Qb}$ of left Hilbert $A$-modules over $M$ and 
\[
\begin{xy}
  \xymatrix{
	Y^{[2]} \times B^{\rm op} \ar[r] \ar[dr] & \pi_2^*\cptEnd{A}{}{\Qb} \ar[d]^{{\rm Ad}_{\phi}}\\
	& \pi_1^*\cptEnd{A}{}{\Qb}
   }
  \end{xy}
\] 
commutes. Thus, we get an induced trivialization $M \times B^{\rm op} \to \cptEnd{A}{}{\widetilde{\Qb}}$.
\end{proof}

\begin{corollary}\label{cor:equivrel}
Stable equivalence is an equivalence relation, which is weaker than isomorphism of Morita bundle gerbes.
\end{corollary}

\begin{proof}\label{pf:equivrel}
To see that $\Hb$ is equivalent to itself, choose $\Qb = \Hb$ and use the isomorphism $\Hb^* \simeq \Hbop$ together with the multiplication $\mu$ on $\Hb$. Similarly it follows that stable equivalence is weaker than isomorphism of Morita bundle gerbes. To see symmetry simply exchange the roles of $\Hb$ and $\Hb'$ and replace $\Qb$ by $\Qb^*$. The hardest part is proving transitivity. Let $\Hb$, $\Hb'$ be equivalent via $\Qb$, let $\Hb'$ and $\Hb''$ be equivalent via $\Qb' \to Y' \times_M Y''$. Let 
\[
\bar{\Qb} = (\pi_{Y \times_M Y'}^*\Qb) \otimes_A (\Delta\Hb') \otimes_A (\pi_{Y' \times_M Y''}^*\Qb')	\to Y \times_M Y' \times_M Y''\ .
\]
On the fibers of $\bar{\Qb}$ we have 
\begin{align*}
	\Qb_{(y,y_1')} \otimes_A \Hb'_{(y_1',y_1')} \otimes_A \Qb'_{(y_1',y'')} & \simeq \Qb_{(y,y_1')} \otimes_A \Hb'_{(y_1',y_2')} \otimes_A \Hb'_{(y_2',y_2')} \otimes_A \Hb'_{(y_2',y_1')} \otimes_A \Qb'_{(y_1',y'')} \\
	& \simeq \Hb_{(y,y)} \otimes_A \Qb_{(y,y_2')} \otimes_A \Hb'_{(y_2',y_2')} \otimes_A \Qb'_{(y_2',y'')} \otimes_A \Hb''_{(y'',y'')} \\
	& \simeq \Qb_{(y,y_2')} \otimes_A \Hb_{(y_2',y_2')} \otimes_A \Qb'_{(y_2',y'')} 	
\end{align*}
where the last isomorphism pulls the outer factors back to the middle one and multiplies them. Using the intertwining properties of $\Qb$ and $\Qb'$ it is straightforward to see that the induced isomorphism $\bar{\pi}_1^*\bar{\Qb} \simeq \bar{\pi}_2^* \bar{\Qb}$ (with $\bar{\pi}_i \colon Y \times_M (Y')^{[2]} \times_M Y'' \to Y \times_M Y' \times_M Y''$ the two canonical projections) satisfies the conditions in corollary \ref{cor:twistedDescent}. Hence it induces a twisted Morita bundle $\Qb \circ \Qb'$ over $Y \times_M Y'$. Note that $(\Qb \circ \Qb')^* \simeq (\Qb')^* \circ \Qb^*$. Since the isomorphism
\[
	\pi_{(Y'')^{[2]}}^*\Hb'' \simeq \pi_1^*\bar{\Qb} \otimes_A \pi_{Y^{[2]}}^*\Hb \otimes_A \pi_2^*\bar{\Qb}^*
\]
commutes with the action used to define $\Qb \circ \Qb'$, it descends to a stable equivalence over $Y \times_M Y''$. 
\end{proof}

\begin{lemma}\label{lem:MoritaInvariance}
If $A$ and $B$ are unital $C^*$-algebras, which are Morita equivalent via a Hilbert $A$-$B$-bimodule $S$, then the latter induces a bijection between the pointed sets
\[
	\Tw{M}{A} \simeq \Tw{M}{B}\ .
\]
In particular for any $n \in \N$ and $S = A^n$ we get the matrix stability of this set, i.e.\ $\Tw{M}{A} \simeq \Tw{M}{M_n(A)}$. \end{lemma}

\begin{proof}\label{pf:}
Let $\Hb$ be a Morita bundle gerbe with respect to $A$ over $M$ and $\Hb'$ be a Morita bundle gerbe with respect to $B$. Then the bijection is induced by the operations
\[
	\Hb \mapsto S^* \otimes_A \Hb \otimes_A S \quad \text{and} \quad \Hb' \mapsto S \otimes_B \Hb' \otimes_B S^*
\]
This descends to stable equivalence classes, since the bundles $\Qb$ can be conjugated with $S$ as well, i.e.\ $S^* \otimes_A \Qb \otimes_A S$. Since $S^* \otimes_A S \simeq B$ and $S \otimes_B S^* \simeq A$, this induces a bijection.
\end{proof}

\subsection{Stable Morita bundle gerbes} 
\label{sub:stable_morita_bundle_gerbes}
In this chapter we are going to discuss the stabilization of Morita bundle gerbes. Since our definition of the twisted tensor product relies on the local triviality of bundles of the form $\trivial{\rtw{H}{\sigma}}$, we have to be careful when dealing with non-unital $C^*$-algebras and Hilbert modules, which are not finitely generated (see remark \ref{rem:Bfin}). Nevertheless, we are still able to define twisted Morita bundles and Morita bundle gerbes for stabilizations of unital algebras just as above. The tensor product used to define the latter yields a locally trivial bundle since the typical fiber is isomorphic to the algebra itself as a left or right module.

\begin{definition}\label{def:stableMorBunGerbe}
Let $A$ be a unital $C^*$-algebra. Denote by $A_s = A \otimes \K$ the stabilization of $A$. A Morita bundle gerbe $\Hb_s$ for $A_s$ will be called \emph{stable Morita bundle gerbe}. Likewise, a twisted Morita bundle for $A_s$ will be called \emph{stable twisted Morita bundle}.	
\end{definition}

\begin{lemma}\label{lem:groupoid}
Every stable Morita bundle gerbe is up to isomorphism of the form given in example \ref{ex:groupoidaction}. Moreover, every twisted Morita $A_s$-$A_s$-bundle over $X$ is of the form $X \times (A_s)_{\sigma}$ for a continuous map $\sigma \colon X \to \Aut{A_s}$.
\end{lemma}

\begin{proof}\label{pf:groupoid}
Let $\Hb_s$ be a stable Morita bundle gerbe. The unitary group $U(A_s) = U(M(A \otimes \K))$ is contractible. This was proven for the norm topology by Mingo in \cite{paper:UnitaryContractible}. The proof for the strict topology is easier, indeed the argument sketched in \cite[exercise 2.M]{book:WeggeOlsen} not only proves the path-connectedness, but also the strict contractibility of $U(M(A \otimes \K))$. Therefore the classifying space of this group is just a point, i.e.\ every bundle of Hilbert $A_s$-modules is trivial. In particular $\Hb_s = Y^{[2]} \times (A_s)_{\lambda}$ for a continuous map $\lambda \colon Y^{[2]} \to \Aut{A_s}$. Since $\mu$ is an isometric isomorphism of trivial left Hilbert $A_s$-module bundles with fiber $A_s$, it is induced by a continuous map $g \colon Y^{[3]} \to U(A_s)$. The associativity implies the condition on $g$ from example \ref{ex:groupoidaction}. 

If $\Qb \to X$ is a twisted $A_s$-$A_s$-Morita bundle, then its typical fibers are isomorphic to $A_s$ as left Hilbert $A_s$-modules, since any Morita self-equivalence of $A_s$ is of the form $(A_s)_{\tau}$ for some element $\tau \in \Aut{A_s}$. Again by the contractibility of $U(A_s)$, it is topologically trivial and the trivialization of the compact operator bundle yields the map $\sigma \colon X \to \Aut{A_s}$.
\end{proof}

Due to the last lemma, the conjugate bundle of a twisted $A_s$-$A_s$-Morita bundle $\Qb = X \times A_{\sigma}$ exists and is isomorphic to $X \times A_{\sigma^{-1}}$. So far, this was not so clear, since corollary~\ref{cor:twistedconjugate} only works for unital algebras. Thus, there is a well-defined notion of stable equivalence of stable Morita bundle gerbes and a pointed set $\Tw{M}{A_s}$. In the stable situation we can apply the \emph{Packer-Raeburn-stabilization trick}, which yields the following result, which is essentially proposition 5.2 in~\cite{paper:HigherCats}.

\begin{lemma}\label{lem:PackerRaeburn}
Suppose $Y$ is an orientable manifold, then every stable Morita bundle gerbe $\Hb_s \to Y^{[2]}$ is isomorphic to one, where the map $g \colon Y^{[3]} \to U(A_s)$ is constantly equal to $1$.
\end{lemma}

\begin{proof}\label{pf:PackerRaeburn}
Let $\mathbb{H} = L^2(Y)$ with respect to the measure induced by some Riemannian metric on~$Y$. Identify $\K$ with the compact operators on this Hilbert space and $U(A_s)$ with the unitary operators $u \colon \HA{A_s} \to \HA{A_s}$. Note that $\HA{A_s} = L^2(Y,A_s)$. By lemma \ref{lem:groupoid}, $\Hb_s = Y^{[2]} \times (A_s)_{\lambda}$ for a pair $(\lambda, g)$. We abbreviate $g(y_1,y_2,y_3)$ by $g_{123}$ and $\lambda(y_1,y_2)$ by $\lambda_{12}$. Let 
\[
	V \colon Y^{[2]} \to U(A_s) \quad ; \quad \left(V(y_1,y_2)f\right)(y_3) = g_{123}^{-1}\cdot f(y_3)\ .
\]
We use the abbreviation $V_{12} = V(y_1,y_2)$. Then $(\lambda_{12}(V_{23})f)(y_4) = \lambda_{12}(g_{234}^{-1})\cdot f(y_4)$ and we have
\begin{align}
	\label{eqn:cocycle}
	(V_{12} \, \lambda_{12}(V_{23})\, g_{123} \cdot f)(y_4) & = g_{124}^{-1}\cdot \lambda_{12}(g_{234}^{-1}) \cdot g_{123} \cdot f(y_4)  \\
	& = g_{134}^{-1}\cdot g_{123}^{-1} \cdot g_{123} \cdot f(y_4) = g_{134}^{-1} \cdot f(y_4) = (V_{13}f)(y_4)\ , \notag
\end{align}
where we applied the condition on $g$ from example \ref{ex:groupoidaction}. Set $\lambda'_{12} = {\rm Ad}_{V_{12}} \circ \lambda_{12}$, then
\begin{align*}
	\lambda'_{12} \circ \lambda'_{23} & = {\rm Ad}_{V_{12}} \circ \lambda_{12} \circ {\rm Ad}_{V_{23}} \circ \lambda_{23} = {\rm Ad}_{V_{12}\lambda_{12}(V_{23})} \circ \lambda_{12} \circ \lambda_{23} \\
	& = {\rm Ad}_{V_{12}\lambda_{12}(V_{23})g_{123}} \circ \lambda_{13} = {\rm Ad}_{V_{13}} \circ \lambda_{13} = \lambda'_{13}\ .
\end{align*}
Therefore $(\lambda', 1)$ satisfies the conditions of example $\ref{ex:groupoidaction}$ and yields a Morita bundle gerbe $\Hb'_s$. The isomorphism $\Hb_s \to \Hb'_s$ is now given by right multiplication with $V^*$ and equation~(\ref{eqn:cocycle}) shows that it intertwines the two multiplication maps.
\end{proof}

\begin{remark}
If we are only interested in stable equivalence classes of Morita bundle gerbes, the condition on $Y$ is no restriction: Up to stable equivalence, every Morita bundle gerbe can be realized over an orientable manifold $Y$ as we will see in the next section, which will also reveal the connection to nonabelian cohomology sets.
\end{remark}

The algebras $A$ and $A_s$ are Morita equivalent via the Hilbert $A_s$-$A$-bimodule $\HA{A}$, where $\mathbb{H}$ is a separable Hilbert space of infinite dimension. The latter is finitely generated and projective as a Hilbert $A_s = A \otimes \K$-module, since $\HA{A} = A \otimes (\K\,e_1) = (A \otimes \K)(1 \otimes e_1)$, where $e_1$ is a rank $1$-projection. Thus, 
\[
	\Hb_s \mapsto (\HA{A}) \otimes_{A_s} \Hb_s \otimes_{A_s} (\HA{A})^*
\]
turns a stable Morita bundle gerbe into an unstable one. But by remark \ref{rem:Bfin} the map
\[
	\Hb \mapsto (\HA{A}) \otimes_A \Hb \otimes_A (\HA{A})^*
\]
is also well-defined and stabilizes a Morita bundle gerbe. Both operations are inverse to each other and descend to stable equivalence classes, since the latter can be (un)stabilized similarly. Thus, we have proven

\begin{lemma}\label{lem:stability}
The matrix stability of $\Tw{M}{A}$ from lemma \ref{lem:MoritaInvariance} extends to a bijection $\Tw{M}{A} \simeq \Tw{M}{A_s}$ induced by the bimodule $\HA{A}$, i.e.\ the stable equivalence classes of Morita bundle gerbes $\Hb$ with respect to $A$ are in $1 : 1$-correspondence with the stable equivalence classes of stable Morita bundle gerbes~$\Hb_s$.
\end{lemma}


\subsection{Nonabelian Cohomology}
In this section we will establish the connection of our results on Morita bundle gerbes with nonabelian cohomology theory, more precisely with the notion of \v{C}ech cohomology with values in a crossed module. We will not dig too deeply into the categorical background of this setup, which can be found for example in \cite{paper:Noohi}, but we will use the notion of cocycles and coboundaries straight away. To define the analogue of the Dixmier-Douady class, we need an additional assumption on the $C^*$-algebras: There is an exact sequence 
\begin{equation} \label{eq:PicSequence}
	1 \to \Inn{A} \to \Aut{A} \to \Pic{A} \ ,
\end{equation}
in which the last map is not surjective in general \cite[Proposition 3.1]{paper:BrownGreenRieffel}. For example for the commutative $C^*$-algebra $C(M)$ the Picard group is a semi-direct product of the homeomorphisms with $H^2(M, \Z)$, whereas $\Aut{A}$ consists just of the homeomorphisms. This motivates the following definition

\begin{definition}\label{def:PicardSurjective}
A $C^*$-algebra will be called \emph{Picard surjective}, if the last map in (\ref{eq:PicSequence}) is surjective.
\end{definition}

\begin{example}
The following $C^*$-algebras are Picard surjective
\begin{itemize}
	\item matrix algebras $M_n(\C)$, in fact $\Pic{M_n(\C)}$ is trivial in this case,
	\item commutative $C^*$-algebras $C_0(X)$, where $H^2(X,\Z) = 0$,
	\item noncommutative tori $A_{\theta}$ if $\theta$ is not quadratic \cite[Corollary 9]{paper:Kodaka},
	\item all stable $C^*$-algebras and the results in this section will hold for stable Morita bundle gerbes \cite{paper:BrownGreenRieffel}
\end{itemize}
This last example is probably the most important, since it will imply that a Morita bundle gerbe always has a Dixmier-Douady class after stabilization. Note that, if $A$ is Picard surjective, then $M_n(A)$ is Picard surjective as well for all $n \in \N$.
\end{example}

\begin{definition}
Let $G$ and $H$ be topological groups and $\alpha \colon G \to H$ be a group homomorphism. The triple $(G,H,\alpha)$ is called a \emph{crossed module} if $H$ acts from the left on $G$, where the action will be denoted by $^h g$ for $h \in H$ and $g \in G$, such that the following conditions are satisfied:
\begin{enumerate}[i)]
\item $^{\alpha(g')} g = g'g(g')^{-1}$
\item $\alpha(^h g) = h\alpha(g)h^{-1}$
\end{enumerate}
\end{definition}

\begin{example}
The adjoint action $\alpha = {\rm Ad} \colon U(A) \to \Aut{A}$ together with the action $^\varphi u = \varphi(u)$ for $\varphi \in \Aut{A}$ and $u \in U(A)$ defines a crossed module.
\end{example}

\begin{definition}\label{def:goodcover}
A cover $U_i \subset M$ is called \emph{good} if all sets $U_i$ as well as all higher intersections $U_{i_1i_2\dots i_n} = \bigcap_{j = 1}^n U_{i_j}$ are either empty or contractible. 
\end{definition}

\begin{definition}\label{def:nonabH1}
Let $\alpha \colon G \to H$ be a crossed module and let $U_i \subset M$ be a good cover of~$M$. A non-abelian \emph{\v{C}ech 1-cocycle} for the crossed module $(G,H,\alpha)$ is a family of maps:
\begin{eqnarray*}
	\lambda_{ij} \colon U_{ij} \to H \ ,\\
	g_{ijk} \colon U_{ijk} \to G
\end{eqnarray*}
such that $\lambda_{ij} \circ \lambda_{jk} = \alpha(g_{ijk}) \lambda_{ik}$ and $\left(^{\lambda_{ij}} g_{jkl}\right)g_{ijl} = g_{ijk}\,g_{ikl}$. Two cocycles $(\lambda_{\bullet}, g_{\bullet})$ over $U_i$ and $(\lambda'_{\bullet}, g'_{\bullet})$ over $U_j'$ are called \emph{equivalent} or \emph{cohomologous}, if there is a good refinement $V_i$ of the cover $U_k \cap U'_l$ and two maps
\begin{eqnarray*}
	r_i \colon V_i \to H \ ,\\
	\vartheta_{ij} \colon V_{ij} \to G
\end{eqnarray*}
such that $\lambda'_{ij} = \alpha(\vartheta_{ij})\,r_i\,\lambda_{ij}\,r_j^{-1}$ and $g'_{ijk} = \left(^{\lambda'_{ij}} \vartheta_{jk}\right)\vartheta_{ij}\,\left(^{r_i}g_{ijk}\right)\,\vartheta_{ik}^{-1}$ over $V_{ijk}$. A cocycle is \emph{normalized} if $\lambda_{ii} = 1$ for all $i \in I$ and $g_{ijk} = 1$ whenever two of the indices agree. A cocycle is called \emph{trivial}, if it is cohomologous to $(\lambda_{\bullet} = 1, g_{\bullet} = 1)$. The pointed set of equivalence classes of cocycles is denoted by $\check{H}^1(M, G \overset{\alpha}{\to} H)$ and will be referred to as the \emph{first non-abelian cohomology} with values in the crossed module $(G,H,\alpha)$. 
\end{definition}

\begin{remark}
Observe that the cocycle condition on $\lambda_{\bullet}$ implies $\lambda_{ii} = \alpha(g_{iii})$. Moreover, the condition on $g_{\bullet}$ yields 
\begin{eqnarray*}
	\left(^{\lambda_{ii}} g_{iij}\right)g_{iij} &=& g_{iii}\,g_{iij} \quad \Rightarrow \quad g_{iij} = g_{iii} \ ,\\
	\left(^{\lambda_{ij}} g_{jjj}\right)g_{ijj} &=& g_{ijj}\,g_{ijj} \quad \Rightarrow \quad g_{ijj} = \left(^{\lambda_{ij}} g_{jjj}\right)\ .
\end{eqnarray*}
Thus, to check that a cocycle is normalized, we only need to see $g_{iji} = 1$ for all $i,j \in I$. Every cocycle is cohomologous to a normalized one, for if we choose an ordering on the index set $I$ and set 
\begin{eqnarray*}
r_i &=& \alpha(g_{iii}) \\
\vartheta_{ij} &=& 
\begin{cases} 
	g_{iji}^{-1}g_{iii}^{-1} & \text{if } i < j \\
	g_{iii}^{-1} & \text{if } i = j \\
	1 & \text{else} 
\end{cases}
\end{eqnarray*}
this yields a coboundary connecting $(\lambda_{\bullet}, g_{\bullet})$ with a normalized cocycle $(\lambda'_{\bullet}, g'_{\bullet})$ if the latter is defined like in the previous paragraph.
\end{remark}

\subsubsection{The non-abelian cohomology class associated to a Morita bundle gerbe}
Every non-abelian \v{C}ech cocycle $(\lambda_{\bullet}, g_{\bullet})$ with values in $(U(A), \Aut{A}, {\rm Ad})$ over a good cover $U_i$ of $M$ gives rise to a Morita bundle gerbe via the following construction: Set
\[
	Y = \coprod_i U_i\ ,
\]
then we have $Y^{[2]}= \coprod_{i,j} U_{ij}$. Thus, $\lambda_{\bullet}$ induces a map $\lambda \colon Y^{[2]} \to \Aut{A}$, likewise $g_{\bullet}$ yields $g \colon Y^{[3]} \to U(A)$. Now set $\Hb_{(\lambda, g)} = Y^{[2]} \times \Atw{\lambda}$. Since $\lambda(y_1,y_2)\,\lambda(y_2,y_3) = {\rm Ad}_{g(y_1,y_2,y_3)}\,\lambda(y_1,y_3)$, the multiplication map $\mu$ over $Y^{[3]}$ is just right multiplication by $g$, where the cocycle condition on $g$ ensures the associativity of $\mu$. 

Suppose $V_k$ with index set $J$ is a refinement of the cover $U_i$ with index set $I$, i.e.\ for every $k \in J$ there is an $i \in I$ with $V_k \subset U_i$. We choose such a value $\iota(k) \in I$ for each $k \in J$ and define 
\[
	Y' = \coprod_{k} V_k \quad , \quad \lambda'_{kl} = \left.\lambda_{\iota(k)\iota(l)}\right|_{V_{kl}} \quad , \quad g'_{klm} = \left.g_{\iota(k)\iota(l)\iota(m)}\right|_{V_{klm}}\ .
\]
We define the \emph{refinement} of $\Hb_{(\lambda, g)}$ with respect to the cover $V_k$ to be the Morita bundle gerbe $\Hb_{(\lambda', g')}$. 

\begin{lemma}\label{lem:stableiso}
If two cocycles $(\lambda_{\bullet}, g_{\bullet})$ and $(\lambda'_{\bullet}, g'_{\bullet})$ are cohomologous, then the associated Morita bundle gerbes $\Hb_{(\lambda, g)}$ and $\Hb_{(\lambda',g')}$ are stably isomorphic.
\end{lemma}

\begin{proof}\label{pf:stableiso}
Suppose first that $\Hb_{(\lambda',g')} \to \left(Y'\right)^{[2]}$ is a refinement of $\Hb_{(\lambda, g)} \to Y^{[2]}$ like above. Since $Y \times_M Y' = \coprod_{i,k} U_i \cap V_k$, we set $\rho_{(i,k)} = \left.\lambda_{i\iota(k)}\right|_{U_i \cap V_k}$, which yields $\rho \colon Y \times_M Y' \to \Aut{A}$. Let $\Qb = (Y \times_M Y') \times A_{\rho}$. This is a twisted Morita bundle over $Y \times_M Y'$. Let $\tau_{(i,k)} = \rho_{(k,i)}$ and $\Qb' = (Y \times_M Y') \times A_{\tau}$. Observe that 
$\Qb^* \simeq \Qb'$ as twisted Morita bundles, where the isomorphism $\Qb' \to \Qb^*$ is induced by right multiplication with $g_{\iota(k)i\iota(k)}g_{\iota(k)\iota(k)\iota(k)}$. Thus, 
\[
	\pi_1^*\Qb \otimes_A \pi_{Y'^{[2]}}^*\Hb_{(\lambda',g')} \otimes_A \pi_2^*\Qb^* \ \simeq \ \pi_1^*\Qb \otimes_A \pi_{Y'^{[2]}}^*\Hb_{(\lambda',g')} \otimes_A \pi_2^*\Qb' \simeq \left(Y \times_M Y'\right)^{[2]} \times A_{\chi}
\]
with $\chi_{(ij,kl)} = \lambda_{i\iota(k)} \lambda_{\iota(k)\iota(l)} \lambda_{\iota(l)j}$. Right multiplication by $g_{i\iota(k)\iota(l)}g_{i\iota(l)j}$ induces an isomorphism of twisted Morita bundles $A_{\chi_{(ij,kl)}} \to A_{\lambda_{ij}}$ and therefore an isomorphism of Morita bundle gerbes
\[
	\pi_1^*\Qb \otimes_A \pi_{Y'^{[2]}}^*\Hb_{(\lambda',g')} \otimes_A \pi_2^*\Qb^* \ \simeq \ \pi_{Y^{[2]}}\Hb_{(\lambda, g)}\ ,
\]
where the compatibility with the bundle gerbe multiplication follows from the cocycle condition.

Now suppose that $(\lambda_{\bullet}, g_{\bullet})$ and $(\lambda'_{\bullet}, g'_{\bullet})$ are cohomologous cocycles connected by a co\-boun\-da\-ry $(r_{\bullet}, \vartheta_{\bullet})$ like in definition \ref{def:nonabH1}. By our previous observation we may assume that everything is defined over the double and triple intersections of a \emph{single} cover $V_k$. In this case $Y = Y'$. Let $\rho_{(i,k)} = r_i \circ \lambda_{ik}$ over $V_i \cap V_k$ and $\Qb = (Y \times_M Y') \times A_{\rho}$. By a similar reasoning as above, $\Qb^*$ is isomorphic to $(Y \times_M Y') \times A_{\tau}$, where $\tau_{(i,k)} = \lambda_{ki} \circ r_i^{-1}$. Therefore
\[
	\pi_1^*\Qb \otimes_A \pi_{Y^{[2]}}^*\Hb_{(\lambda,g)} \otimes_A \pi_2^*\Qb^* \ \simeq \ \pi_1^*\Qb \otimes_A \pi_{Y^{[2]}}^*\Hb_{(\lambda,g)} \otimes_A \pi_2^*\Qb' \simeq \left(Y \times_M Y'\right)^{[2]} \times A_{\chi}
\]
with $\chi_{(ij,kl)} = r_i \lambda_{ik} \lambda_{kl} \lambda_{lj} r_j^{-1} = r_i \alpha(g_{ikl}\,g_{ilj}) \lambda_{ij} r_j^{-1} = \alpha(r_i(g_{ikl}\,g_{ilj}) \vartheta_{ij}^{-1}) \lambda'_{ij}$. Right multiplication by $r_i(g_{ikl}\,g_{ilj})\,\vartheta_{ij}^{-1}$ therefore yields an isomorphism
\[
	\pi_1^*\Qb \otimes_A \pi_{Y^{[2]}}^*\Hb_{(\lambda,g)} \otimes_A \pi_2^*\Qb^* \ \simeq \ \pi_{Y'^{[2]}}^*\Hb_{(\lambda', g')}\ .
\]
To see that this intertwines the bundle gerbe multiplications on $\Hb_{(\lambda,g)}$ and $\Hb_{(\lambda',g')}$, let $\bar{\lambda}_{ij} = r_i \lambda_{ij} r_j^{-1}$ and $\bar{g}_{ijk} = r_i(g_{ijk})$ and note that the right multiplication with $r_i(g_{ikl}\,g_{ilj})$ identifies $\left(Y \times_M Y'\right)^{[2]} \times A_{\chi}$ with $\pi_{Y'^{[2]}}^*\Hb_{(\bar{\lambda}, \bar{g})}$ in a compatible way, due to the cocycle condition on $g_{ijk}$. By the definition of $g'_{ijk}$ the following diagram commutes:
\[
  \begin{xy}
   \xymatrix{
	A_{\left(r_i\lambda_{ij}r_j^{-1}\right) \left(r_j \lambda_{jk} r_k^{-1}\right)}  \ar[rr]^{\quad \quad \cdot \vartheta_{ij}^{-1}\lambda'_{ij}(\vartheta_{jk}^{-1})} \ar[d]^{\cdot r_i(g_{ijk})} & & A_{\lambda'_{ij}\lambda'_{jk}} \ar[d]^{\cdot g'_{ijk}}\\
	A_{r_i \lambda_{ik} r_k^{-1}} \ar[rr]^{\cdot \vartheta_{ik}^{-1}} & & A_{\lambda'_{ik}}
   }
  \end{xy}
\]
Thus, right multiplication by $\vartheta_{ij}^{-1}$ induces an isomorphism $\Hb_{(\bar{\lambda}, \bar{g})} \simeq \Hb_{(\lambda', g')}$ as Morita bundle gerbes.
\end{proof}

If $A$ is Picard surjective, any Morita self-equivalence is of the form $A_{\sigma}$ for $\sigma \in \Aut{A}$. In particular the typical fibers of a Morita bundle gerbe $\Hb$ for $A$ or of a twisted $A$-$A$-Morita bundle $\Qb$ are isomorphic to $A$ as left Hilbert $A$-modules. We will see that this enables us to associate a non-abelian cohomology class to $\Hb$ by the following construction, which uses the next easy lemma.
\begin{lemma}\label{lem:isoAut}
Let $H$, $H'$ be $A$-$B$-Morita equivalences, which are isomorphic as left Hilbert $A$-modules via $\theta \colon H \to H'$, then there is an automorphism $\varphi_{\theta} \in \Aut{B}$ depending on $\theta$, such that 
\[
	H \ \simeq \ H'_{\varphi_{\theta}}
\]
via $h \mapsto \theta(h)$ as Hilbert $A$-$B$-bimodules, where the right action of $B$ on $H'_{\varphi_{\theta}}$ is twisted by $\varphi_{\theta}$.
\end{lemma}

\begin{proof}\label{pf:isoAut}
Let $\Psi_H \colon B^{\rm op} \to \cptEnd{A}{}{H}$ and $\Psi_{H'} \colon B^{\rm op} \to \cptEnd{A}{}{H'}$ be the two algebra isomorphisms defining the right action of $B$ on $H$, respectively $H'$. If $H'$ is now twisted with
\[
\varphi_{\theta} \colon B^{\rm op} \overset{\Psi_H}{\longrightarrow} \cptEnd{A}{}{H} \overset{{\rm Ad}_{\theta}}{\longrightarrow} \cptEnd{A}{}{H'} \overset{\Psi_{H'}^{-1}}{\longrightarrow} B^{\rm op}
\]
then the isomorphism $\theta$ intertwines the two actions of $B$ on $H$ and $H'_{\varphi_{\theta}}$. Using 
\[
	\Psi_H(\scalB{x}{y})(z) = \scalA{z}{x}{y}
\]
we also see that $\theta$ preserves both inner products.
\end{proof}

Let $\Hb$ be a Morita bundle gerbe for a Picard surjective $C^*$-algebra $A$. Choose a good cover $U_i$ of $M$ such that there exist sections $\sigma_i \colon U_i \to Y$ of $\pi \colon Y \to M$. Note that $m \mapsto (\sigma_i(m), \sigma_j(m))$ yields a well-defined map to $Y^{[2]}$ and set $\Hb_{ij} = (\sigma_i, \sigma_j)^*\Hb$. This is a twisted $A$-$A$-Morita bundle over $U_{ij}$. Choose trivializations $\kappa_{ij} \colon \Hb_{ij} \to \trivial{A}$ as left Hilbert $A$-module bundles. By lem\-ma~\ref{lem:isoAut} there is a map $\lambda_{ij} \colon U_{ij} \to \Aut{A}$ induced by $\kappa_{ij}$ such that $(\Hb_{ij})_x \simeq \Atw{\lambda_{ij}(x)}$ as a Hilbert $A$-$A$-bimodule. Over the triple intersections we have the multiplication map $\mu \colon \Hb_{ij} \otimes_A \Hb_{jk} \to \Hb_{ik}$, which induces the following isomorphism of twisted Morita bundles
\[
 \Atw{\lambda_{ij}\lambda_{jk}} \to \Atw{\lambda_{ij}} \otimes_A \Atw{\lambda_{jk}} \to \Hb_{ij} \otimes_A \Hb_{jk} \to \Hb_{ik} \to \Atw{\lambda_{ik}}\ .
\]
Since it is in particular an isomorphism of left Hilbert $A$-module bundles, it is given by right multiplication with a unitary operator. Therefore it corresponds to a function $g_{ijk} \colon U_{ijk} \to U(A)$. Necessarily, we have $\lambda_{ij}\lambda_{jk} = {\rm Ad}_{g_{ijk}}\lambda_{ik}$. Due to the associativity of $\mu$, the following diagram commutes
\[
\begin{xy}
	\xymatrix{
		\Atw{\lambda_{ij}\lambda_{jk}\lambda_{kl}} \ar[r] \ar@{=}[d] & \Atw{\lambda_{ij}\lambda_{jk}} \otimes_A \Atw{\lambda_{kl}}  \ar[r]^{\cdot g_{ijk} \otimes \id{}} & \Atw{\lambda_{ik}} \otimes_A \Atw{\lambda_{kl}} \ar[r]^/0.8em/{\cdot g_{ikl}}& \ar@{=}[d] \Atw{\lambda_{il}} \\
		\Atw{\lambda_{ij}\lambda_{jk}\lambda_{kl}} \ar[r] & \Atw{\lambda_{ij}} \otimes_A \Atw{\lambda_{jk}\lambda_{kl}}  \ar[r]^{\cdot \id \otimes g_{jkl}} & \Atw{\lambda_{ij}} \otimes_A \Atw{\lambda_{jl}} \ar[r]^/0.8em/{\cdot g_{ijl}}& \Atw{\lambda_{il}}
		}
\end{xy}
\]
which implies the cocycle identity. Thus, we end up with a cocycle $\omega(\Hb) = (\lambda_{\bullet}, g_{\bullet}) \in \check{H}^1(M, U(A) \to \Aut{A})$, which we will call the \emph{non-abelian Dixmier Douady class}. We still need to check that $\omega(\Hb)$ is independent of the choices of $\sigma_i$ and $\kappa_{ij}$ in its cohomology class. This will follow from the next lemma.

\begin{lemma}\label{lem:omegaHb}
Let $A$ be a Picard surjective $C^*$-algebra. The class $\omega(\Hb)$ is independent of all choices up to coboundaries and the assignment $\Hb \to \omega(\Hb)$ induces a $1:1$ correspondence of $\check{H}^1(M, U(A) \to \Aut{A})$ with the stable equivalence classes of Morita bundle gerbes, i.e.
\[
	\Tw{M}{A}\ \simeq\ \check{H}^1(M, U(A) \to \Aut{A})\ .
\]
\end{lemma}
\begin{proof}
In fact, the proof that $\omega$ is well-defined on stable equivalence classes and that its class is independent of choices both rest upon the same argument: If $U_i'$ is another good cover of $M$ with sections $\sigma_i' \colon U_i' \to Y$, pullback bundles $\Hb'_{ij}$ and trivializations $\kappa_{ij}' \colon \Hb'_{ij} \to \Atw{\lambda_{ij}'}$ for automorphisms $\lambda_{ij}'$ induced by $\kappa_{ij}'$, then we choose a good refinement $V_l$ of $U_i \cap U_j'$. Let $\Qb_l = (\sigma_l', \sigma_l)^*\Hb$. This is a twisted Morita bundle over $V_l$ and due to remark \ref{rem:fiberIsA} we have $\Qb_l^* \simeq (\sigma_l, \sigma_l')^*\Hb$ as twisted Morita bundles. Therefore the bundle gerbe multiplication induces an isomorphism
\[
	\Hb'_{kl} \simeq \Qb_k \otimes_A \Hb_{kl} \otimes_A \Qb_l^*\ .
\]
Now choose trivializations $\kappa_{kl} \colon \Hb_{kl} \to \Atw{\lambda_{kl}}$, $\kappa'_{kl} \colon \Hb'_{kl} \to \Atw{\lambda'_{kl}}$ and $\rho_k \colon \Qb_k \to \Atw{r_k}$. The induced isomorphism $\Atw{\lambda'_{kl}} \to \Atw{r_k\lambda_{kl}r_l^{-1}}$ is again given by right multiplication with a unitary $\vartheta_{kl} \in U(A)$ and we have 
\[
	\lambda'_{kl} = {\rm Ad}_{\vartheta_{kl}}\,r_k\,\lambda_{kl}\,r_l^{-1}
\]
Since $\Hb'_{kl} \to \Qb_k \otimes_A \Hb_{kl} \otimes_A \Qb_l^*$ is an isomorphism of Morita bundle gerbes and therefore intertwines the multiplication maps we have a commutative square,
\[
\begin{xy}
	\xymatrix{
		\Atw{\lambda'_{kl}\lambda'_{lm}} \ar[r]^{\cdot g'_{klm}} \ar[d]_{\cdot \lambda'_{kl}(\vartheta_{lm})\vartheta_{kl}} & \Atw{\lambda'_{km}} \ar[d]^{\cdot \vartheta_{km}} \\
		\Atw{\lambda_{kl}\lambda_{lm}} \ar[r]^{\cdot g_{klm}} & \Atw{\lambda_{km}}
 	}
\end{xy}
\]
which proves that the two cocycles are actually cohomologous. 

It remains to prove that $\omega$ is well-defined and injective on stable equivalence classes. So suppose $\Hb \to Y^{[2]}$ and $\Hb' \to \left(Y'\right)^{[2]}$ are two Morita bundle gerbes, which are stably equivalent via $\Qb \to Y \times_M Y'$. Choose good covers $U_i$ and $U_i'$ of $M$ for $\Hb$ and $\Hb'$ like above and refine $U_i \cap U_j'$ to a good cover $V_l$, for which there exist sections $\theta_l \colon V_l \to Y \times_M Y'$, let $\sigma_l \colon V_l \to Y$ and $\sigma_l' \colon V_l \to Y'$ be the induced sections of $Y$ and $Y'$. Set $\Qb_l = \theta_l^*\Qb$, $\Hb_{kl} = (\sigma_k, \sigma_l)^*\Hb$ and define $\Hb'_{kl}$ similarly. The stable equivalence now yields an isomorphism of twisted Morita bundles $\Hb'_{kl} \to \Qb_k \otimes_A \Hb_{kl} \otimes_A \Qb_l^*$. Now proceed exactly as above to see that the corresponding cocycles are cohomologous. 

By lemma \ref{lem:stableiso}, the injectivity on stable isomorphism classes follows if we can prove that $\Hb$ is stably isomorphic to some Morita bundle gerbe of the form $\Hb_{(\lambda,g)}$, where $(\lambda_{\bullet}, g_{\bullet})$ is a representative for $\omega(\Hb)$. So, let $U_i$ be a good cover, $(\lambda_{\bullet}, g_{\bullet})$ be a representative of $\omega(\Hb)$ over $U_i$, let $\sigma_i$ and $\kappa_{ij}$ be as above and consider
\[
	\rho_i \colon \pi^{-1}(U_i) = Y \times_M U_i \to Y^{[2]} \quad ; \quad (y,x) \mapsto (y, \sigma_i(x))
\]
Let $Y' = \coprod_i U_i$, then $Y \times_M Y' = \coprod_{i} Y \times_M U_i$ and the collection of the $\rho_i$ yields a map $\rho \colon Y \times_M Y' \to Y^{[2]}$. Let $\Qb = \rho^*\Hb$, then $\id{\Hb} \otimes \kappa_{ij}^{-1} \otimes \id{\Hb}$ and $\mu$ yield isomorphisms
\[
	\rho_i^*\Hb \otimes_A A_{\lambda_{ij}} \otimes_A \rho_j^*\Hb^* \ \simeq \ \rho_i^*\Hb \otimes_A \Hb_{ij} \otimes_A \rho_j^*\Hb^*\ \simeq \ \left.\Hb\right|_{\pi^{-1}(U_{ij})}\ ,
\]
which fit together to form
\[
	\pi_1^*\Qb \otimes_A \pi_{Y'^{[2]}}^* \Hb_{(\lambda, g)} \otimes_A \pi_2^*\Qb^*\ \simeq \ \pi_{Y^{[2]}}^*\Hb\ .
\]
By the definition of $g_{ijk}$ this isomorphism is compatible with the multiplication map~$\mu$.
\end{proof}

\begin{corollary}\label{cor:stableinvariance}
Let $A$ be a unital $C^*$-algebra (not necessarily Picard surjective). Then we have bijections of pointed sets
\[
	\Tw{M}{A} \simeq \Tw{M}{A_s} \simeq \check{H}^1(M, U(A_s) \to \Aut{A_s}) \ .
\]
Moreover, there is a surjection of pointed sets
\begin{equation} \label{eq:AutAs}
	\check{H}^1(M, \Aut{A_s}) \to \Tw{M}{A_s}\ ,
\end{equation}
which sends a principal $\Aut{A_s}$-bundle $P \to M$ to the Morita bundle gerbe $\Hb \to P^{[2]}$ from example \ref{ex:groupaction}, i.e.\ every element in $\Tw{M}{A}$ arises from a principal $\Aut{A_s}$-bundle.
\end{corollary}

\begin{proof}\label{pf:stableinvariance}
The first bijection was established in lemma \ref{lem:stability} and is given by stabilization, the second is lemma \ref{lem:omegaHb} via $[\Hb_s] \mapsto \omega(\Hb_s)$. It remains to prove the surjectivity of the map in the second claim. By the proof of lemma~\ref{lem:omegaHb} and lemma \ref{lem:PackerRaeburn} a stable Morita bundle gerbe $\Hb_s$ is isomorphic to one of the form $\Hb_{(\lambda,1)}$ for an $\Aut{A_s}$-cocycle $[\lambda] \in \check{H}^1(M, \Aut{A_s})$ over a good cover $U_i$. Let $P$ be the principle $\Aut{A_s}$-bundle described by $\lambda$, then the bundle which is build like in example \ref{ex:groupaction} from $P$ is stably equivalent to $\Hb_{(\lambda,1)}$.
\end{proof}

Note that $\check{H}^1(M, \Aut{M_n(A)})$ maps to $\Tw{M}{A} = \Tw{M}{A_s}$ using the construction in example \ref{ex:groupaction}. This map commutes with stabilization in the following sense
\[
  \begin{xy}
   \xymatrix{
	\check{H}^1(M, \Aut{M_n(A)}) \ar[r] \ar[d]& \Tw{M}{A} \ar@{=}[d] \\
	\check{H}^1(M, \Aut{A_s}) \ar[r]^{(\ref{eq:AutAs})}& \Tw{M}{A}
   }
  \end{xy}
\]
where the vertical map on the left hand side is induced by $\Aut{M_n(A)} \to \Aut{A_s}$, which sends an automorphism $\sigma$ of $A \otimes M_n(\C)$ to $\sigma \otimes \id{\K} \in \Aut{A \otimes M_n(\C) \otimes \K} \simeq \Aut{A \otimes \K}$. Choosing different isomorphisms $M_n(\C) \otimes \K \to \K$ changes $\Aut{M_n(A)} \to \Aut{A_s}$ by conjugation with an inner automorphism, which has no effect on the isomorphism classes of principal $\Aut{A_s}$-bundles.

\begin{definition}\label{def:matrixstable}
A class $[\Hb] \in \Tw{M}{A}$ is called \emph{matrix stable} if there exists an $n \in \N$, such that $[\Hb]$ lies in the image of the map $\check{H}^1(M, \Aut{M_n(A)}) \to \Tw{M}{A}$. We call a bundle of $C^*$-algebras $\Ab \to M$ with fiber $M_n(A)$ affiliated with $\Hb$ if the isomorphism class of its principal $\Aut{M_n(A)}$-bundle in $\check{H}^1(M, \Aut{M_n(A)})$ is mapped to $[\Hb]$.
\end{definition}

\begin{example}\label{ex:MatrixStableForMnC}
If $A = \C$, then $\Tw{M}{\C} \simeq \check{H}^1(M, U(\mathbb{H}) \to PU(\mathbb{H})) \simeq \check{H}^2(M, U(1)) \simeq H^3(M,\Z)$ is a group and (\ref{eq:AutAs}) is an isomorphism. A class $[\Hb]$ in this group is matrix stable if its classifying map $\varphi$ factors as 
\[
	\varphi \colon M \to BPU(n) \to BPU(\mathbb{H}) \simeq K(\Z,3)\ .
\]
It was shown in \cite{paper:AtiyahSegal} that this happens, precisely if $[\Hb] \in H^3(M,\Z)$ is torsion. Thus, we can see matrix stable elements in $\Tw{M}{A}$ as a generalization of torsion classes, even though this set carries \emph{no group structure} in general.
\end{example}

\subsection{Nonabelian Bundle Gerbes and Principal Bibundles} 
\label{ssub:nonabelian_bundle_gerbes_and_principal_bibundles}
Another interpretation of nonabelian cohomology is in terms of nonabelian bundle gerbes, which are based on a theory of principal bibundles \cite{paper:JurcoNonAbelian}. In this section we will show how Morita bundle gerbes fit into the conceptual picture of nonabelian bundle gerbes.

\begin{definition}
Let $\alpha \colon G \to H$ be a crossed module. A \emph{$(G,H,\alpha)$-crossed module bundle} is a principal $G$-bundle $P$ together with a trivialization $\Psi \colon P \to H$ of the principal $H$ bundle $P \times_{\alpha} H$ associated to $P$ via $\alpha$.
\end{definition}

\begin{remark} 
We may think about $(G,H,\alpha)$-bundles in terms of principal $G$-$G$-bibundles, where the left action of $G$ is induced by $\Psi$ via
\[
	g \cdot p = p\cdot ^{\Psi(p)}\!g
\]
Indeed we have
\begin{eqnarray*}
g' \cdot (pg) &=& (pg) \cdot ^{\Psi(pg)}\!g' = pg\cdot ^{\alpha(g^{-1})} \left(^{\Psi(p)}g'\right) \\
&=& pg\,g^{-1} \left(^{\Psi(p)}g'\right) g = p \left(^{\Psi(p)}g'\right) g = (g' \cdot p) \cdot g\ ,
\end{eqnarray*}
therefore the left and right actions commute. If we start with a principal $G$-$G$-bibundle $P$, we get $\Psi$ via $\Psi(p) = h$, where $h \in \Aut{G}$ is such that $g \cdot p = p\cdot\,^h\!g$. The $(G,H,\alpha)$-crossed module bundles correspond to those bibundles for which $\Psi(p) \in H$ corresponds to $\alpha$ by the construction above (for further details see \cite{paper:JurcoNonAbelian}).
\end{remark}

\begin{example}\label{ex:FrameBundle}
Let $\Qb \to X$ be a twisted $A$-$A$-Morita bundle over a connected space $X$ with typical fiber $H$, which we will view as a right Hilbert $A$-module bundle with a trivialization $\Psi \colon X \times A \to \cptEnd{}{A}{\Qb}$. Note that $\cptEnd{}{A}{H} = A$ and $\U{}{A}{H} = U(A)$. As a locally trivial right Hilbert $A$-module bundle there is a principal $U(A)$-bundle $\Pb \to X$ associated to it. At a point $x \in X$ the fiber of $\Pb$ consists of all right $A$-linear isometric isomorphisms $H \to \Qb_x$ and the right action of $U(A)$ is given by pre-concatenation. Likewise, there is a principal $\Aut{A}$-bundle induced by $\cptEnd{}{A}{\Qb}$, more precisely this is given by $\Pb \times_{\rm Ad} \Aut{A}$, since 
\[
	\cptEnd{}{A}{\Qb} = \Pb \times_{\rm Ad} \cptEnd{}{A}{H} = (\Pb \times_{\rm Ad} \Aut{A}) \times_{\tau} A\ ,
\]
where $\tau$ is the canonical action of $\Aut{A}$ on $A$. The map $\Psi$ induces a trivialization of the latter bundle and therefore turns $\Pb$ into a principal $U(A)$-$U(A)$-bibundle. In comparison with the very similar construction for vector bundles we call $\Pb$ the \emph{frame bibundle} of $\Qb$.

On the other hand let $\Pb \to X$ be a principal $U(A)$-$U(A)$-bibundle and let $H$ be a Hilbert $A$-$A$-bimodule inducing a Morita self-equivalence of $A$. Since the group $U(A)$ acts canonically on $H$ from the left via $\varrho$, we can associate a twisted Morita bundle $\Qb = \Pb \times_{\varrho} H$ to $\Pb$ using the right action of $U(A)$ on $\Pb$. Then the trivialization of $\Pb \times_{\rm Ad} \Aut{A}$ yields a map $\Psi \colon A \times X \to \cptEnd{}{A}{\Qb}$, which is for each $x \in X$ an isomorphism of $C^*$-algebras. Isomorphic twisted Morita bundles necessarily have fibers, which are isomorphic \emph{as right Hilbert $A$-modules}, therefore  
\end{example}

\begin{lemma}\label{lem:IsoClasses}
Isomorphism classes of twisted Morita bundles $\Qb \to X$ over a connected space $X$ are in $1:1$-cor\-res\-pon\-dence with pairs $([\Pb], [H])$, where $\Pb$ is a principal $U(A)$-$U(A)$-bibundle, $[\Pb]$ denotes its isomorphism class, $H$ is a Morita self-equivalence of $A$ and $[H]$ denotes its isomorphism class as a right Hilbert $A$-module.
\end{lemma}

Let $\Pb$ and $\Pb'$ be two principal $U(A)$-$U(A)$-bibundles over $X$. As explained in \cite{paper:JurcoNonAbelian}, their product $\Pb \cdot \Pb'$ is defined to be the quotient of $\Pb \times \Pb'$ with respect to the right action $(p, p') \mapsto (pu, u^{-1}p')$ for $u \in U(A)$. Of course the two remaining actions turn $\Pb \cdot \Pb'$ into a principal bibundle again. Let $\Qb = \Pb \times_{\varrho} A$ and $\Qb' = \Pb' \times_{\varrho} A$, then we have

\begin{lemma}\label{lem:twistedProduct}
The twisted tensor product $\Qb \otimes_A \Qb'$ is isometrically isomorphic to $(\Pb \cdot \Pb') \times_{\varrho} A$, i.e.\ $([\Pb], [A]) \otimes_A ([\Pb'], [A]) = ([\Pb \cdot \Pb'], [A])$ under the correspondence of lemma \ref{lem:IsoClasses}.
\end{lemma}

\begin{proof}\label{pf:twistedProduct}
We will denote elements in $\Qb$ by $[p,v]$ for $p \in \Pb$ and $v \in A$, likewise for $\Qb'$. Since $\Pb'$ is a principal bibundle, there is a map $\Psi \colon \Pb' \to \Aut{A}$, which is equivariant in the sense that for $p' \in \Pb'$ and $u \in U(A)$ we have $\Psi(p'u) = {\rm Ad}_{u^*} \circ \Psi(p')$. Likewise, we have $\Phi \colon \Pb \to \Aut{A}$. The left action of $A$ on $\Qb'$ is given by $a \cdot [p', v'] = [p', \Psi(p')(a)v']$. Let
\[
	\Theta \colon \Qb \otimes_A \Qb' \to (\Pb \cdot \Pb') \times_{\varrho} A \quad ; \quad [p,v] \otimes [p',v'] \mapsto [(p,p'), \Psi(p')(v)v'] \ .
\]
To check that this is well-defined, let $u \in U(A), a \in A$ and observe
\begin{align*}
	\Theta([pu,v] \otimes [p',v']) & = [(pu,p'), \Psi(p')(v)v'] = [(p,p' \cdot \Psi(p')(u)), \Psi(p')(v)v'] \\
	& = [(p,p'), \Psi(p')(uv)v'] = \Theta([p,uv] \otimes [p',v']) \\
	\Theta([p,v] \otimes [p'u,v']) & = [(p,p'u), \Psi(p'u)(v)v'] = [(p,p'u), u^*\Psi(p')(v)uv'] \\ 
	& = [(p,p'), \Psi(p')(v)uv'] = \Theta([p,v] \otimes [p',uv']) \\
	\Theta([p,v] \cdot a \otimes [p',v']) & = [(p,p'), \Psi(p')(va)v'] = [(p,p'), \Psi(p')(v) \Psi(p')(a) v'] \\ 
	& = \Theta([p,v] \otimes a \cdot [p',v']) 
\end{align*}
Thus, $\Theta$ is well-defined on the algebraic tensor product in each fiber. But it also preserves the right $A$-valued inner product due to the following calculation
\begin{align*}
	\scalR{[p,v_1] \otimes [p',v_2]}{[p,w_1] \otimes [p',w_2]}{A} & = \scalR{v_2}{\Psi(p')(\scalR{v_1}{w_1}{A})\,w_2}{A} \\
& = \scalR{\Psi(p')(v_1)\,v_2}{\Psi(p')(w_1)\,w_2}{A} \\
& = \scalR{\Theta([p,v_1] \otimes [p',v_2])}{\Theta([p,w_1] \otimes [p',w_2])}{A}
\end{align*}
and therefore extends to a (partial) isometry on the fibers. By the open mapping theorem it suffices to see that $\Theta$ is surjective. Let $e_n$ be an approximate unit for $A$ and let $[(p,p'),w] \in (\Pb \cdot \Pb') \times_{\varrho} A$. Let $a_n = [p,e_n] \otimes [p',w]$. Since
\[
	\lVert [p,e_n] \otimes [p',w] - [p,e_m] \otimes [p',w] \rVert = \lVert e_n\,\Psi(p')^{-1}(w) - e_m\,\Psi(p')^{-1}(w) \rVert\ ,
\]
this is a Cauchy sequence in the fiber over the point onto which $p$ and $p'$ project. Since the fibers are complete, it converges. Note that 
\[
	\Theta(a_n) = [(p,p'), \Psi(p')(e_n)w] = [(p,p'), \Psi(p')(e_n\, \Psi(p')^{-1}(w))]
\]
converges to $[(p,p'),w]$. Since $\Theta$ is continuous, we are almost done. At last, we need to check that $\Theta$ is an isomorphism of twisted Morita bundles, i.e.\ that the following diagram commutes
\[
  \begin{xy}
   \xymatrix{
	X \times A \ar[rr]^{\!\!\!\alpha} \ar@{=}[d] & & \cptEnd{}{A}{\Qb \otimes \Qb'} \ar[d]^{\beta}\\
	X \times A \ar[rr]^{\!\!\!\gamma} & & \cptEnd{}{A}{(\Pb \cdot \Pb') \times_{\varrho} A}
   }
  \end{xy}
\]
Since $\alpha$ sends $(x,a)$ to the operator $[p,v] \otimes [p',v'] \mapsto [p, \Phi(a)v] \otimes [p',v']$, the composition $\beta \circ \alpha$ sends $(x,a)$ to the compact operator, which maps $[(p,p'),w]$ to $[(p,p'), \Psi(p') \circ \Phi(p)(a)w]$. This is precisely the left multiplication by $a$ on the fiber of $(\Pb \cdot \Pb') \times_{\varrho} A$, which finishes the proof.
\end{proof}

\begin{definition}\label{def:nonabBG}
A nonabelian bundle gerbe in the sense of \cite{paper:JurcoNonAbelian, paper:JurcoCrossed} for the crossed module $U(A) \to \Aut{A}$ is a principal $U(A)$-bibundle $\Pb \to Y^{[2]}$ together with a bibundle isomorphism $\pi_{12}^* \Pb \cdot \pi_{23}^* \Pb \to \pi_{13}^* \Pb$, which satisfies the bibundle analogue of associativity as in definition \ref{def:MoritaBG}, i.e.\ the following diagram commutes:
\[
  \begin{xy}
   \xymatrix{
	\left(\pi_{12}^* \Pb \cdot \pi_{23}^*\Pb\right) \cdot \pi_{34}^*\Pb \ar@{=}[rr] \ar[d] & & \pi_{12}^* \Pb \cdot \left(\pi_{23}^*\Pb \cdot \pi_{34}^*\Pb\right) \ar[d] \\
	\pi_{13}^* \Pb \cdot \pi_{34}^*\Pb \ar[dr] & & \pi_{12}^* \Pb \cdot \pi_{24}^*\Pb \ar[dl]\\
	& \pi_{14}^* \Pb & 
   }
  \end{xy}
\]	
\end{definition}

The construction of the frame bibundle is functorial in the sense that an isomorphism of twisted Morita bundles $\varphi \colon \Qb \to \Qb'$ induces a corresponding isomorphism of the frame bibundles. More precisely, $\varphi$ yields an isomorphism $H \to H'$ of the corresponding typical fibers and maps an isometry $H \to \Qb_x$ to the isometry $H' \to H \to \Qb_x \to \Qb_{\varphi(x)}$. Since $\varphi$ intertwines the two trivializations, the induced map intertwines the two left multiplications on the frame bibundle. This works vice versa \emph{if} the fibers are isomorphic to $A$ as right Hilbert $A$-modules. Thus, we have

\begin{corollary}\label{cor:nonabBunGer}
Isomorphism classes of nonabelian bundle gerbes $\Pb \to Y^{[2]}$ for the crossed module $U(A) \to \Aut{A}$ are in $1:1$-correspondence with Morita bundle gerbes $\Hb \to Y^{[2]}$ with typical fibers isomorphic to $A$ as right Hilbert $A$-modules via 
\[
	\Pb \mapsto \Pb \times_{\varrho} A\ . 
\] 	
\end{corollary}

In case the algebra $A$ is Picard surjective, this yields a bijective correspondence between Morita bundle gerbes and nonabelian bundle gerbes. Apart from this, the situation is more complicated, since the cokernel of the homomorphism $\Aut{A} \to \Pic{A}$ becomes important. In fact, one can see the coset decomposition of $\Pic{A}$ with respect to ${\rm Out}(A)$ (the image of $\Aut{A}$ in $\Pic{A}$) in lemma \ref{lem:IsoClasses}. In case $A$ is a simple $C^*$-algebra carrying a unique trace~$\tau$ separating the equivalence classes of projections then the cokernel is the fundamental group
\[
	\mathcal{F}(A) = \left\{ (\text{tr} \otimes \tau)(p)\ |\ p \in M_n(A) = M_n(\C) \otimes A,\ p = p^2 = p^*, \ pM_n(A)p \simeq A \right\}
\]
as defined in \cite{paper:NawataWatatani}. In this case
\[
	1 \to ZU(A) \to U(A) \to \Aut{A} \to \Pic{A} \to \mathcal{F}(A) \to 1
\] 
is exact. So, the non-exactness stems from the self-similarity of $A$ in this case. Nevertheless, this information is captured by Morita equivalences and therefore by Morita bundle gerbes as well. 

\section{Twisted Operator K-Theory} 
\label{sec:twisted_operator_k_theory}
Modules over bundle gerbes were used in \cite{paper:KTheoryBGM} to give a description of twisted $K$-theory (with torsion twist) in terms of objects as closely related to vector bundles as possible. In this chapter we are going to extend the results from \cite{paper:KTheoryBGM} to the case of Hilbert $A$-module bundles with twists given by Morita bundle gerbes introduced above. Even though the pointed set $\Tw{M}{A}$ is not a group anymore, most of the main results transfer to this more general setup. 

\begin{definition}\label{def:MoritaModule}
Let $A$ be a unital $C^*$-algebra, $M$ be a manifold, $\pi \colon Y \to M$ be a fibration and let $\Hb \to Y^{[2]}$ be a Morita bundle gerbe for $A$ over $Y^{[2]}$. A \emph{twisted Morita module} $\Eb \to Y$ for $\Hb$ is a locally trivial bundle of right Hilbert $A$-modules $\Eb$ with typical fiber $H$ and structure group $\U{}{A}{H}$ together with an isometric isomorphism of right Hilbert $A$-module bundles
\[
	\gamma \colon \pi_1^*\Eb \otimes_A \Hb \to \pi_2^*\Eb\ ,
\]
such that the following associativity diagram commutes
\[
  \begin{xy}
   \xymatrix{
	\left(\pi_1^*\Eb \otimes_A \pi_{12}^*\Hb\right) \otimes_A \pi_{23}^*\Hb \ar@{=}[rr] \ar[d]^{\gamma \otimes \id{}} & & \pi_1^* \Eb \otimes_A \left(\pi_{12}^*\Hb \otimes_A \pi_{23}^*\Hb\right) \ar[d]^{\id{} \otimes \mu} \\
	\pi_2^* \Eb \otimes_A \pi_{23}^*\Hb \ar[dr]^{\gamma} & & \pi_1^* \Eb \otimes_A \pi_{13}^*\Hb \ar[dl]^{\gamma}\\
	& \pi_{3}^* \Eb & 
   }
  \end{xy}
\]
A \emph{(compact) morphism} of twisted Morita modules $\Eb \to Y$ and $\Eb' \to Y$ for the same Morita bundle gerbe $\Hb$ is a fiberwise bounded $A$-linear adjointable (compact) homomorphism of right Hilbert $A$-module bundles $\varphi \colon \Eb \to \Eb'$, such that the following diagram commutes
\[
  \begin{xy}
   \xymatrix{
	\pi_1^*\Eb \otimes_A \Hb \ar[rr]^{\ \ \gamma} \ar[d]^{\pi_1^*\varphi \otimes \id{\Hb}} & & \pi_2^* \Eb \ar[d]^{\pi_2^*\varphi} \\
	\pi_1^*\Eb' \otimes_A \Hb \ar[rr]^{\ \ \gamma'} & & \pi_2^* \Eb'
   }
  \end{xy}
\]
A twisted Morita bundle $\Eb \to Y$ will be called \emph{finitely generated} if the fibers are finitely generated as right Hilbert $A$-modules and likewise for \emph{projective}. The bundle of compact adjointable $A$-linear operators $\Eb_y \to \Eb'_y$ will be denoted by $\cpthom{}{A}{\Eb}{\Eb'}$ with $\cptEnd{}{A}{\Eb} = \cpthom{}{A}{\Eb}{\Eb}$. Since we are dealing mostly with finitely generated projective Hilbert $A$-modules, for which the compact and the bounded adjointable operators agree, we only consider compact operators in the following.  
\end{definition}

\begin{corollary}\label{cor:cptDescent}
Let $\Eb \to Y$ be a twisted Morita module for $\Hb$. There is a a bundle isomorphism
\[
	\phi \colon \pi_2^*(\cptEnd{}{A}{\Eb}) \to \pi_1^*(\cptEnd{}{A}{\Eb})\ ,
\]
which satisfies the condition of lemma \ref{lem:descentlemma}, such that the global sections of the induced bundle $\mmm{}{A}{\Eb}$ over $M$ correspond to the compact Morita module endomorphisms of $\Eb$.
\end{corollary}

\begin{proof}\label{pf:cptDescent}
Let $(y_1, y_2) \in Y^{[2]}$ and $T \colon \Eb_{y_1} \to \Eb_{y_1}$ be a compact adjointable operator, then we define $\phi^{-1}(T) = \gamma \circ (\pi_1^* T \otimes \id{\Hb}) \circ \gamma^{-1}$, i.e.\ 
\[
	\phi^{-1}(T) \colon \Eb_{y_2} \to \Eb_{y_1} \otimes_A \Hb_{(y_1,y_2)} \overset{T \otimes \id{\Hb}}{\longrightarrow} \Eb_{y_1} \otimes_A \Hb_{(y_1,y_2)} \to \Eb_{y_2}
\]
on the fibers. This satisfies the associativity condition due to the following commutative diagram:
\[
  \begin{xy}
   \xymatrix{
	\Eb_{y_3} \ar[r] \ar@{=}[d]& \Eb_{y_2} \otimes_A \Hb_{(y_2,y_3)} \ar[r] & \Eb_{y_1} \otimes_A \Hb_{(y_1,y_2)} \otimes_A \Hb_{(y_2,y_3)} \ar[r]^{T \otimes \id{\ }}  \ar[d]^{\id{\Eb_{y_1}} \otimes \mu} & \Eb_{y_1} \otimes_A \Hb_{(y_1,y_2)} \otimes_A \Hb_{(y_2,y_3)} \ar[d]^{\id{\Eb_{y_1}} \otimes \mu}\\
	\Eb_{y_3} \ar[rr] & & \Eb_{y_1} \otimes \Hb_{(y_1,y_3)} \ar[r]^{T \otimes \id{\ }}& \Eb_{y_1} \otimes \Hb_{(y_1,y_3)}
   }
  \end{xy}
\]
where the square on the left hand side commutes due to the associativity of the action $\gamma$. A global section $\tau \colon M \to \mmm{}{A}{\Eb}$ corresponds to a section $\hat{\tau} \colon Y \to \cptEnd{}{A}{\Eb}$, such that $\varphi(\hat{\tau}(y_2)) = \hat{\tau}(y_1)$, which translates into the condition for morphisms of twisted Morita modules.
\end{proof}

\begin{example}\label{ex:trivial}
Let $\Hb$ be the Morita bundle gerbe from example \ref{ex:groupaction} for the unital $C^*$-al\-ge\-bra~$A$, i.e.\ $\Hb = P^{[2]} \times A_{\varphi \circ g}$ for some principal $G$-bundle $P$ and a homomorphism $\varphi \colon G \to \Aut{A}$. The trivial bundle of free rank $n$ Hilbert $A$-modules 
\[
	\Eb = P \times A^n
\]
is a twisted Morita module via the action
\[
	\gamma \colon \pi_1^*\Eb \otimes_A \Hb \to \pi_2^*E \quad ; \quad (p_1, p_2, v \otimes a) \mapsto (p_1, p_2, \varphi(g(p_1,p_2))^{-1}(va))
\]
The bundle of Morita module endomorphisms turns out to be $\mmm{}{A}{\Eb} = P \times_{\varphi} M_n(A)$, where $\varphi$ acts on $M_n(A) = M_n(\C) \otimes A$ as $\id{M_n(\C)} \otimes \varphi$.
\end{example}

\begin{example}\label{ex:local}
To have a local model for Morita modules, we need the following example. Let $P \to M$ be a trivializable principal $G$-bundle with an equivariant map $\tau \colon P \to G$. Let $\varphi \colon G \to \Aut{A}$ be a homomorphism. Let $\Hb$ be the Morita bundle gerbe from example \ref{ex:groupaction} constructed from $P$ and $\varphi$. Let $H = t\,A^n$ be a finitely generated projective right Hilbert $A$-module, where $t \in M_n(A)$ is a projection and define
\[
	\Eb = P \times \rtw{H}{\varphi \circ \tau}\ .
\]
This is in fact a twisted Morita module with respect to the action
\[
	\gamma \colon \pi_1^*\Eb \otimes_A \Hb \to \pi_2^*\Eb \quad ; \quad (p_1,p_2,v \otimes a) \mapsto (p_1,p_2,v \cdot a)\ , 
\]
where the dot denotes the action \emph{in the fiber over $p_1$}. By lemma \ref{lem:loctriv} it is indeed a locally trivial bundle of right Hilbert $A$-modules, which is isomorphic to 
\[
	\Eb' = \left\{ (p,v) \in P \times A^n\ | \ \varphi(\tau(p))^{-1}(t)v = v \right\}
\]
as a twisted Morita module by lemma \ref{lem:emb}, where the action of $\Hb$ on $\Eb'$ is given just like in the previous example. Thus, $\Eb$ embeds into the trivial twisted Morita module. 
\end{example}

Let $\Eb' \to Y'$ be a twisted Morita module for $\Hb'$ and let $\bar{\Eb} = \pi_{Y'}^*\Eb' \otimes_A \Qb \otimes_A \Delta\Hb$. On the fibers we have the following isomorphisms of right Hilbert $A$-modules
\begin{align*}
	\Eb'_{y_2'} \otimes \Qb_{(y_2',y)} \otimes \Hb_{(y,y)} & \simeq \Eb'_{y_1'} \otimes \Hb'_{(y_1',y_2')} \otimes \Qb_{(y_2',y)} \otimes \Hb_{(y,y)} \\ 
	& \simeq \Eb'_{y_1'} \otimes  \Qb_{(y_1',y)} \otimes \Hb_{(y,y)} \otimes \Hb_{(y,y)} \\
	& \simeq \Eb'_{y_1'} \otimes  \Qb_{(y_1',y)} \otimes \Hb_{(y,y)} 
\end{align*}
Let $\pi'_i \colon Y'^{[2]} \times_M Y \to Y' \times_M Y$ be the canonical projections. The associativity conditions on $\Hb$ and $\Eb$ and the compatibility of (\ref{eqn:stableiso}) with the product ensure that the induced iso $\pi_2^{'*}\bar{\Eb} \to \pi_1^{'*} \bar{\Eb}$ satisfies the condition for lemma \ref{lem:descentlemma}. Therefore we get an induced bundle of right Hilbert $A$-modules $\Qb(\Eb') \to Y$. Furthermore $\Hb$ acts from the right on $\bar{\Eb}$ as follows: 
\begin{align*}
	\Eb'_{y'} \otimes \Qb_{(y',y_1)} \otimes \Hb_{(y_1,y_1)} \otimes \Hb_{(y_1,y_2)} & \simeq \Eb'_{y'} \otimes \Qb_{(y',y_1)} \otimes \Hb_{(y_1,y_2)} \otimes \Hb_{(y_2,y_2)} \\ 
	& \simeq \Eb'_{y'} \otimes \Hb'_{(y',y')} \otimes \Qb_{(y',y_2)} \otimes \Hb_{(y_2,y_2)} \\
	& \simeq \Eb'_{y'} \otimes \Qb_{(y',y_2)} \otimes \Hb_{(y_2,y_2)}  
\end{align*}
By the intertwining properties of $\Qb$ and the associativity of the multiplication and the action, we see that this action commutes with the isomorphism used to define the bundle $\Qb(\Eb')$. Thus, we end up with a twisted Morita module $\Qb(\Eb') \to Y$. The notation of this bundle is due to its compatibility with the transitivity of stable equivalence: We have $(\Qb \circ \Qb')(\Eb'') = \Qb(\Qb'(\Eb''))$, where we used the notation from the proof of corollary~\ref{cor:equivrel}. The construction of $\Qb(\Eb')$ is functorial with respect to morphisms of Morita bundle gerbes, since these are precisely the maps that respect the action of $\Hb'$. Moreover, $\Qb^*(\Qb(\Eb')) \simeq \Eb' \otimes \Delta\Hb' \simeq \Eb'$ and we have proven

\begin{theorem}\label{thm:transfer}
If $\Hb \to Y^{[2]}$ and $\Hb' \to (Y')^{[2]}$ are two Morita bundle gerbes, which are stably equivalent via $\Qb \to Y \times_M Y'$. Then the maps $\Eb \mapsto \Qb^*(\Eb)$ and $\Eb' \mapsto Q(\Eb')$ provide a natural equivalence between the category of Morita modules for $\Hb$ over $Y$ and the category of Morita modules for $\Hb'$ over $Y'$. This equivalence is also bijection when restricted to finitely generated projective modules.
\end{theorem}

The direct sum of the action maps turns the direct sum of two twisted Morita modules into another one. All operations so far are compatible with forming those direct sums, in particular the functor $\Qb$ is additive. 
\begin{definition}\label{def:twistedK}
Let $A$ be a unital $C^*$-algebra and $\Hb$ a Morita bundle gerbe for $A$. We define the twisted $K$-group $K^0_{\Hb}(M)$ with twist $\Hb$ to be the Grothendieck group of the isomorphism classes of twisted Morita modules. 
\end{definition}

The next corollary is a direct consequence of the previous theorem:
\begin{corollary}\label{cor:twistedKdepends}
The twisted $K$-group $K^0_{\Hb}(M)$ depends up to non-canonical isomorphism induced by stable equivalences only on the class of $\Hb$ in $\Tw{M}{A}$.
\end{corollary} 

In order to give an operator algebraic description of $K^0_{\Hb}(M)$ for the case of matrix stable twists $[\Hb] \in \Tw{M}{A}$, we first need to prove that every finitely generated projective Morita module has a stable inverse.

\begin{lemma}\label{lem:embTrivial}
Let $\Hb$ be the bundle gerbe from example \ref{ex:groupaction}. Let $\Eb \to P$ be a finitely generated and projective twisted Morita module for the Morita bundle gerbe $\Hb$. Then there exists an $N \in \N$ and a finitely generated projective twisted Morita module $\Fb \to P$, such that 
\[
	\Eb \oplus \Fb \ \simeq \ \trivial{A^N}\ ,
\] 
where the right hand side denotes the trivial Morita module as explained in example~\ref{ex:trivial}.
\end{lemma}

\begin{proof}\label{pf:embTrivial} 
Without loss of generality we may assume $M$ to be connected. Choose a good finite open cover $\bigcup_{i \in I}U_i \supset M$ and define $P_i = \left.P\right|_{U_i}$,  such that there exist local sections $\sigma_i \colon U_i \to P$ for the projection $\pi \colon P \to M$ and equivariant maps $\tau_i \colon P_i \to G$ with $\tau_i \circ \sigma_i = 1$. Let $\Eb_i = \sigma_i^*\Eb$ and choose an isomorphism $\varphi_i \colon \Eb_i \to U_i \times H$ with $H = tA^n$ for some projection $t \in M_n(A)$. Let $\trivial{\rtw{H}{\varphi \circ \tau_i}}$ be the Morita module from example \ref{ex:local} and define
\[
	\Delta_i \colon P_i \to P_i^{[2]}\quad ; \quad p \mapsto (\sigma_i(\pi(p)),p)\ .
\]
Let $\gamma \colon \pi_1^*\Eb \otimes_A \Hb \to \pi_2^*\Eb$ and $\delta \colon \pi_1^*\trivial{\rtw{H}{\varphi \circ \tau_i}} \otimes_A \Hb \to \pi_2^*\trivial{\rtw{H}{\varphi \circ \tau_i}}$ be the action maps. These induce an isometric isomorphism of right Hilbert $A$-module bundles 
\[
	\psi_i \colon \left.\Eb\right|_{P_i} \overset{\gamma^{-1}}{\longrightarrow} \pi^*\Eb_i \otimes_A \Delta_i^*\Hb \overset{\varphi_i \otimes \id{\Hb}}{\longrightarrow} \trivial{H} \otimes_A \Delta_i^*\Hb \overset{\delta}{\longrightarrow} \trivial{\rtw{H}{\varphi \circ \tau_i}}\ ,
\]
where the third map uses the identification of $\trivial{\rtw{H}{\varphi \circ \tau_i}}$ with $\trivial{H}$ over the image of $\sigma_i$. Using the associativity of $\Hb$, it is easy to check that this is an isomorphism of twisted Morita modules.

Now let $h_i^2 \colon M \to [0,1]$ be a partition of unity subordinate to the cover $U_i$. Let $k = |I|$. $N = n \cdot k$ and set
\[
	\Psi \colon \Eb \to \trivial{A^N} \quad ; \quad v \mapsto \sum_{i \in I} h_i(\pi(p)) \cdot \psi_i(v)
\]
for $v \in \Eb_p$. Where we suppressed the embedding of $\trivial{\rtw{H}{\varphi \circ \tau_i}}$ into the trivial bundle in the notation and $\psi_i$ is understood as an embedding into the $i$th summand of $\trivial{A^N} = \bigoplus_{i} \trivial{A^n}$. Now define
\[
	\Fb = \left\{ (p,v) \in \trivial{A^N}\ \mid \ \scalR{v}{\Psi(w)}{A} = 0\ \forall w \in E_p \right\}\ .
\]
It is a priori not clear that this defines a locally trivial bundle of right Hilbert $A$-modules over $P$ and that we have $\Eb \oplus \Fb = \trivial{A^N}$, but we can use the same machinery as in \cite[Theorem~2.14]{paper:SchickL2Index} to show that this is indeed the case. Since $\Psi$ is an isometric embedding of twisted Morita modules and intertwines the two actions, we see that $\delta \colon \pi_1^*\trivial{A^N} \otimes_A \Hb \to \pi_2^*\trivial{A^N}$ restricts to an isometric isomorphism of $\Fb$, therefore the latter is a twisted Morita module and the isometric isomorphism $\Eb \oplus \Fb \to \trivial{A^N}$ given by $\Psi$ on the first summand and by the embedding on the second is a morphism of twisted Morita modules.
\end{proof}

\begin{theorem}\label{thm:Yeah}
Let $\Hb$ be a representative of a matrix stable element $[\Hb] \in \Tw{M}{A}$ and let $\Ab \to M$ be a bundle of $C^*$-algebras affiliated with $\Hb$. Then we have
\[
	K^0_{\Hb}(M) \ \simeq\ K_0(C(M,\Ab))\ .
\]
\end{theorem}

\begin{proof}\label{pf:Yeah}
By corollary \ref{cor:twistedKdepends} it suffices to give an isomorphism $K^0_{\Hb}(M) \simeq K_0(C(M,\Ab))$ for the case when $\Hb$ is of the form given in example \ref{ex:groupaction} since any matrix stable element can be obtained in that way. We will denote the fiber of $\Ab$ by $A$. Let $P$, $\varphi$ and $\Hb$ be as in example~\ref{ex:groupaction}. If $\Eb \to P$ is a twisted Morita module for $\Hb$, then by lemma \ref{lem:embTrivial} it embeds as a direct summand into $\trivial{A^N}$. Thus, there is a twisted Morita morphism
\[
	t_{\Eb} \colon \trivial{A^N} \to \trivial{A^N}\ ,
\] 
which projects to the fiber of $\Eb$ at every point. By corollary \ref{cor:cptDescent}, it corresponds to a projection in $C(M, \mmm{}{A}{\Eb}) \simeq C(M, M_N(\Ab))$. This construction is well-behaved with respect to direct sums, hence we get a homomorphism
\[
		K^0_{\Hb}(M) \to K_0(C(M,\Ab)) \quad ; \quad [\Eb] - [\Eb'] \mapsto [t_{\Eb}] - [t_{\Eb'}] \ .
\]
To see the inverse map, let $t \in C(M,M_N(\Ab))$ be a projection-valued section of $M_N(\Ab)$, which corresponds to a $G$-equivariant map $P \to M_N(A)$, i.e.\ $t(pg) = \varphi(g)^{-1}(t(p))$. Let 
\[
	\Eb_t = \left\{ (p,v) \in P \times A^N\ |\ t(p)v = v \right\}\ .
\]
Using the fact that norm-close projections are unitarily equivalent - like we already did in lemma \ref{lem:loctriv} - it is easy to see that $\Eb_t$ is a locally trivial bundle of finitely generated projective right Hilbert $A$-modules. It carries an action of $\Hb$ given by 
\[
	\gamma \colon \pi_1^*E \otimes_A \Hb \to \pi_2^*\Eb \quad ; \quad (p_1,p_2, v \otimes a) \mapsto \varphi(g(p_1,p_2))^{-1}(va)\ ,
\]
which indeed maps into $\pi_2^*\Eb$, since we have for $g_{12} = g(p_1,p_2)$:
\[
	t(p_2)\varphi(g_{12})^{-1}(va) = t(p_1g_{12})\varphi(g_{12})^{-1}(va) = \varphi(g_{12})^{-1}(t(p_1)va) = \varphi(g_{12})^{-1}(va)\ .
\]
Thus, $\Eb_t$ is a twisted Morita module. It is clear that these two constructions are inverse to each other in the sense that 
\[
	K_0(C(M,\Ab)) \to K^0_{\Hb}(M) \quad ; \quad [t] - [t'] \mapsto [\Eb_t] - [\Eb_{t'}]
\]
provides an inverse to the above homomorphism.
\end{proof}

\begin{example}
If $P \to M$ is a principal $PU(n) = \Aut{M_n(\C)}$-bundle, then the above construction boils down to the results in \cite{paper:KTheoryBGM}. More precisely: As we have seen, a Morita bundle gerbe can be turned into a regular $S^1$-bundle gerbe by tensoring with the canonical Morita equivalence $\C^n$ between $\C$ and $M_n(\C)$, i.e.
\[
	\Hb \mapsto \C^{n*} \otimes_{M_n(\C)} \Hb \otimes_{M_n(\C)} \C^n\ .
\]
Likewise, a twisted Morita module $\Eb \to P$ can be turned into a regular bundle gerbe module via 
\[
	\Eb \mapsto \Eb \otimes_{M_n(\C)} \C^n\ . 
\]
This is particularly interesting for the following twist: If $M$ is an oriented Riemannian manifold and $P$ the frame bundle associated to its tangent bundle, then we have a homomorphism
\[
	SO(n) \to \Aut{\C l(n)}\ .
\] 
Using the construction from example \ref{ex:groupaction}, we get an element $[\Hb] \in \Tw{M}{\C l(n)} \simeq H^3(M, \Z)$, which corresponds to the Bockstein of the second Stiefel-Whitney class $W_3(M) = \beta(w_2(M))$.
\end{example}

\section{Twisted $K$-theory via Cuntz Algebras}
The Cuntz algebra $\Cuntz{n}$ is the universal $C^*$-algebra generated by $n$ partial isometries, i.e.\ by elements $s_1, \dots s_n$ with relations $s_i^*s_i = 1$ and $\sum_{j=1}^n s_js_j^* = 1$ \cite{paper:CuntzAlg}. It was proven in \cite{paper:KTheoryOfCuntz} that $K_0(\Cuntz{n+1}) \simeq \Z/n\Z$ and $K_1(\Cuntz{n+1}) \simeq 0$. Therefore we can consider $K_0(C(M,\Cuntz{n+1}))$, the $K$-theory of the $C^*$-algebra of $\Cuntz{n+1}$-valued continuous functions. Evaluating on a point we see that this cohomology theory has the same coefficients as $K^0(M, \Z/n\Z)$ defined via homotopy theory. It follows from a description of the latter by Karoubi \cite[Exercise~6.18]{book:Karoubi} and the K\"unneth theorem for $K$-theory of operator algebras that these two actually agree. Thus, we get an operator algebraic description of $K$-theory with $\Z/n\Z$-coefficients. 

This has the advantage that $\Cuntz{n}$ carries a rich automorphism structure \cite{paper:AutoCuntz1, paper:AutoCuntz2}, which allows us to twist it. For example, the unitary group $U(n)$ acts on the generators in an obvious way via 
\[
	s_i \mapsto \sum_{j=1}^n u_{ij} s_j \ .
\]
It can be checked that this action extends to an automorphism of $\Cuntz{n}$ and that $U(n) \to \Aut{\Cuntz{n}}$ is continuous with respect to the point-norm topology. Therefore any principal $U(n)$-bundle $P$ over $M$ yields a twist $\Hb$ of $K$-theory with $\Z/(n-1)\Z$-coefficients and theorem~\ref{thm:Yeah} gives a geometric description of this group by twisted Morita modules. On the operator algebraic side we get $\Ab = P \times_{U(n)} \Cuntz{n}$ and $K_0(C(M,\Ab))$.

This gets even more interesting if we replace $\Cuntz{n}$ by the infinite Cuntz algebra $\Cuntz{\infty}$. $K$-theory has a priori more twists than just those classified by $H^3(M,\Z)$. The full classifying space is $BBU_{\otimes}$. Thus, the question arises, if there are geometric descriptions of these higher twists. The infinite Cuntz algebra seems like a good starting point. So, the question here would be: Are automorphisms of $\Cuntz{\infty}$ related to higher twistings of $K$-theory? 

Other applications might arise if we consider fields of $C^*$-algebras instead of locally trivial bundles. This would be closer to the work of Echterhoff, Nest and Oyono-Oyono \cite{paper:Echterhoff}.

\bibliographystyle{plain} 
\bibliography{Bundles}

\vspace{0.5cm}
\textsc{\small Mathematisches Institut, Westf\"alische Wilhelms-Universit\"at M\"unster, Ein\-stein\-stra\ss e 62, 48149 M\"unster, Germany}\\[0.1cm]
\textit{\small E-mail: u.pennig@uni-muenster.de}
\end{document}